\documentclass[a4paper,11pt]{amsart}
\usepackage{bm}
\usepackage[square,numbers]{natbib}
\usepackage{rotating}
\usepackage{xparse}
\usepackage{hyperref}
\usepackage{mathtools}
\usepackage{enumerate}
\usepackage{tabularx}
\usepackage{bbm}
\usepackage{geometry}
\usepackage{amsmath,amssymb,amsthm,amsfonts}
\newtheorem{theorem}{Theorem}[section]
\newtheorem{corollary}[theorem]{Corollary}
\newtheorem{lemma}[theorem]{Lemma}
\newtheorem{proposition}[theorem]{Proposition}
\theoremstyle{definition}
\newtheorem{definition}[theorem]{Definition}
\newtheorem{example}[theorem]{Example}
\newtheorem{remark}[theorem]{Remark}
\ExplSyntaxOn
\DeclareExpandableDocumentCommand{\IfNoValueOrEmptyTF}{mmm}
 {
  \IfNoValueTF{#1}{#2}
   {
    \tl_if_empty:nTF {#1} {#2} {#3}
   }
 }
\ExplSyntaxOff
\DeclareFontEncoding{LS1}{}{}
\DeclareFontSubstitution{LS1}{stix}{m}{n}

\let\pmold\pm
\newcommand\cev[1]{\overleftarrow{#1}}
\newcommand{\alphabet}{\mathcal{A}}
\DeclareDocumentCommand\pv{s o}{\IfBooleanTF{#1}{{\cev{\mathsf{p}}}}{\mathsf{p}}\IfNoValueF{#2}{^{(#2)}}}
\DeclareDocumentCommand\tv{s o}{\IfBooleanTF{#1}{{\cev{\mathsf{q}}}}{\mathsf{q}}\IfNoValueF{#2}{^{(#2)}}}
\DeclareDocumentCommand\tree{o}{T\IfNoValueF{#1}{^{(#1)}}}
\DeclareDocumentCommand\tshift{o o}{\mathcal{T}\IfNoValueOrEmptyTF{#1}{}{^{(#1)}}_{\IfNoValueF{#2}{#2}}}
\DeclareDocumentCommand\aprod{o o}{X\IfNoValueOrEmptyTF{#1}{}{^{\times #1}}_{\IfNoValueF{#2}{#2}}}
\DeclareDocumentCommand\level{o m o}{\Xi\IfNoValueF{#1}{^{(#1)}}_{#2\IfNoValueF{#3}{:#3}}}
\DeclareDocumentCommand\lattice{o m o}{\Delta\IfNoValueF{#1}{^{(#1)}}_{#2\IfNoValueF{#3}{:#3}}}
\DeclareDocumentCommand\block{m o}{B_{#1\IfNoValueF{#2}{:#2}}}
\DeclareDocumentCommand\pblock{m o}{Z_{#1\IfNoValueF{#2}{:#2}}}
\DeclareDocumentCommand\dv{m o}{\tau_{#1\IfNoValueF{#2}{:#2}}}
\DeclareDocumentCommand\dvset{m o}{D_{#1\IfNoValueF{#2}{:#2}}}
\newcommand{\norm}[1]{|#1|}
\newcommand{\Norm}[1]{\left\Vert#1\right\Vert}
\DeclareDocumentCommand\pm{s o}{\IfBooleanTF{#1}{{\cev{\mathsf{P}}}}{\mathsf{P}}\IfNoValueF{#2}{^{(#2)}}}
\DeclareDocumentCommand\tm{s o}{\IfBooleanTF{#1}{{\cev{\mathsf{Q}}}}{\mathsf{Q}}\IfNoValueF{#2}{^{(#2)}}}
\DeclareDocumentCommand\qm{s o}{\IfBooleanTF{#1}{{\cev{\mathsf{R}}}}{\mathsf{R}}\IfNoValueF{#2}{^{(#2)}}}
\DeclareDocumentCommand\trm{m o}{\eta_{#1\IfNoValueF{#2}{:#2}}}
\DeclareDocumentCommand\trmset{m o}{S_{#1\IfNoValueF{#2}{:#2}}}

\DeclareDocumentCommand\proddom{s m o}{\IfBooleanTF{#1}{W}{W}_{#2\IfNoValueF{#3}{:#3}}}

\newcommand{\onevec}{\mathbbm{1}}
\newcommand{\stdvec}[1]{\mathsf{e}_{#1}}

\newcommand{\DKL}{D_{\mathrm{KL}}}
\newcommand{\alphpv}{\Gamma}
\newcommand{\alphsm}{\Upsilon}
\newcommand{\incmat}{\mathbf{M}}
\newcommand{\transmat}{M}
\newcommand{\initvec}{\pi}
\newcommand{\obsfunc}{A}

\title[Hausdorff dimension of irreducible Markov hom tree-shifts]{Hausdorff dimensions of irreducible Markov hom tree-shifts}

\author[Jung-Chao Ban]{Jung-Chao Ban}
\address[Jung-Chao Ban]{Department of Mathematical Sciences, National Chengchi University, Taipei 11605, Taiwan, ROC.}
\address{Math. Division, National Center for Theoretical Science, National Taiwan University, Taipei 10617, Taiwan. ROC.}
\email{jcban@nccu.edu.tw}
\author[Guan-Yu Lai]{Guan-Yu Lai}
\address[Guan-Yu Lai]{Department of Mathematical Sciences, National Chengchi University, Taipei 11605, Taiwan, ROC.}
\email{gylai@nccu.edu.tw}
\author[Yu-Liang Wu]{Yu-Liang Wu}
\address[Yu-Liang Wu]{Department of Mathematical Sciences, P.O. Box 3000, 90014 University of Oulu, Finland}
\email{Yu-Liang.Wu@oulu.fi}
\usepackage{color}

\begin{document}

\begin{abstract}
    This paper features a Cram\'er's theorem for finite-state Markov chains indexed by rooted $d$-trees, obtained via the method of types in the classical analysis of large deviations. Along with the theorem comes two applications: an almost-sure type convergence of sample means and a formula for the Hausdorff dimension of the symbolic space associated with the irreducible Markov chain.
\end{abstract}
\maketitle
\section{Introduction}
This paper presents a version of Cram\'er's theorem for finite-state Markov chains indexed by $d$-trees, accompanied by a number of folklore theorems and an exploration of their connections to dimension theory.

Let $\alphabet$ be a state space that is at most countable. Recall that a classical Markov chain is an $\alphabet$-valued stochastic process $(X_n)_{n \in \mathbb{N}}$ satisfying the \emph{Markov property}: For any $n \in \mathbb{N}$ and $a, b \in \alphabet$, the finite-dimensional distributions of the process satisfy
\[
\mathbb{P}(X_n=b|X_{n-1}=a, X_{n-2}, \cdots, X_0) = \mathbb{P}(X_n=b|X_{n-1}=a) = \transmat_{a,b}
\]
with the \emph{transition probabilities} $\{\transmat_{a,b}: a,b \in \alphabet\}$ independent of $n$. In this context, the transition probabilities satisfy $\sum_{b \in \alphabet} \transmat_{a,b} = 1$ for all $a \in \alphabet$, forming a stochastic matrix $\transmat = (\transmat_{a,b})_{a,b \in \alphabet}$ known as the \emph{transition matrix}. Additionally, the law of the initial state $X_0$ can be represented as a probability vector $\initvec = (\initvec_a)_{a \in \alphabet} = (\mathbb{P}(X_0 = a))_{a \in \alphabet}$, referred to as the \emph{initial distribution}. For example, a lattice random walk can be modeled as a Markov chain whose state space $\alphabet$ consists of all nodes in the lattice and $X_n$ is the position at time $n$ with the transition probability $M_{a,b}$ specifying the probability to move from position $a$ to $b$. Extending this paradigm, Benjamini and Peres in \cite{Benjamini1994} introduced the concept of generalized Markov chains as $\alphabet$-valued processes $(X_g)_{g \in \tree}$ indexed by some generalized ``time'' $\tree$. In this framework, $\tree$ is a \emph{rooted tree}, that is, a locally finite, connected acyclic graph that features a distinguished vertex $\epsilon$ called \emph{root}. For $g, h \in \tree$, we write $g \le h$ if $g$ lies on the (unique) path from the root $\epsilon$ to $h$. We denote by $\tilde{g}$ the parent of $g$ with $\tilde{g} \le g$, and define $g \wedge h$ as the farthest vertex on both paths from $\epsilon$ to $h$ and from $\epsilon$ to $g$. A tree-indexed Markov chain is then defined as follows.
\begin{definition} \label{def:Markov-chain}
    Let $\alphabet$ denote the state space, $\tree$ a rooted tree, $\transmat \in \mathbb{R}_{\ge 0}^{\alphabet \times \alphabet}$ a stochastic matrix, and $\initvec \in \mathbb{R}_{\ge 0}^{\alphabet}$ a probability vector. A \emph{Markov chain indexed by $\tree$} with transition matrix $\transmat$ and initial distribution $\initvec$ is an $\alphabet$-valued stochastic process $(X_g)_{g \in \tree}$ satisfying $\mathbb{P}(X_\epsilon = a) = \initvec_a$ for all $a \in \alphabet$ and the Markov property: For all $a, b \in \alphabet$,
    \begin{equation} \label{eq:Markov_assumption}
        \mathbb{P}(X_{g}=b|X_{\tilde{g}}=a, X_{h} \text{ for } g \wedge h \le \tilde{g})=\mathbb{P}(X_{g} = b|X_{\tilde{g}} = a)=\transmat_{a, b},
    \end{equation} 
\end{definition}
\noindent Notably, the underlying probability space $(\Omega, \mathcal{F}, \mathbb{P})$ of a $\tree$-indexed Markov chain can be realized using a standard Borel space. Specifically, by defining $\Omega = \alphabet^{\tree}$, $\mathcal{F}$ as the Borel $\sigma$-algebra, and $X_g(t) = t_g$, the Markov property together with the assumption of initial distribution uniquely determines a Borel probability measure (denoted also by $\mathbb{P}$ by a slight abuse of notation). This measure, referred to as a \emph{Markov measure}, is supported within 
\begin{equation} \label{eq:tree-shifts}
    \tshift[][\incmat] = \{t \in \alphabet^{\tree}: \incmat_{t_{\tilde{g}},t_g} = 1 \text{ for all } g \in \tree \setminus \{\epsilon\}\},
\end{equation}
where $\incmat$ is the \emph{incidence matrix} associated with $\transmat$, a matrix satisfying $\incmat_{a,b} = 0$ or $1$ depending on $\transmat_{a,b} = 0$ or $> 0$, respectively. In particular, when $\tree = \mathbb{Z}_+$, the measure $\mathbb{P}$ reduces to a classical Markov measure and the space $\tshift[][\incmat]$ to a Markov subshift, making them special cases of the broader framework. Motivated by this observation, the present study examines and compares tree-indexed and conventional Markov chains through the lenses of large deviation theory and dimension theory, and our focus is on the class of finite-state Markov chains indexed by rooted $d$-trees ($d \ge 2$), each corresponding to the Cayley graph of the free semigroup generated by $\Sigma = \{1,2,\cdots,d\}$:
\[
\tree=\bigcup_{i=0}^{\infty} \Sigma^i := \{\epsilon\} \cup \bigcup_{i=1}^{\infty} \{g_1 g_2 \cdots g_i: g_j \in \Sigma, \forall 1 \le j \le i \},
\]
with $\Sigma^0=\{\epsilon\}$ being part of our convention.

Despite their introduction in \cite{Benjamini1994} in a general form, special cases of tree-indexed Markov chains had appeared in earlier studies, typically as special cases of i.i.d.~tree-indexed processes. For example, first-passage percolation on trees (cf.~\cite{Lyons1992}) is modeled by tree-indexed i.i.d.~random variables where each variable represents the ``passage time'' to traverse between vertices in $\tree$, and the primary focus is on the asymptotic passage time to reach a vertex from the root. For developments in this area, see for example \cite{Benjamini1994a} and see \cite{Fan2001} for a multifractal analysis of first-passage percolation. Further discussion on tree-indexed i.i.d.~variables and some related fields can be found in \cite[Section 1]{Benjamini1994a} and \cite{Benjamini1996}.

Beyond the i.i.d.~setting, the tree-indexed Markov chains, as explored in \cite{Benjamini1994a}, are closely related to tree-indexed random walks on groups or general countable graphs, where the rooted tree plays the role of ``generalized time'' as illustrated in the previous paragraph. This approach addresses a key challenge in the analysis of random walks on large groups or graphs, where a single sample path visits only a relatively small portion of the space and thus provides limited information. With the abundance of sample paths offered by the generalized time, tree-indexed random walks provide a more comprehensive view of their underlying structure.

Another frequently explored aspect of non-conventional Markov chains is their entropy, which motivates the first part of the present work. This line of research originates from the study of statistical mechanics on trees, as initiated in \cite{Preston1974} and \cite{Spitzer1975}. For related studies, see also \cite{Brightwell2002,burton1995variational,Eggarter1974,Georgii2011}. Continuing in this direction, the article focuses on both the typical and large deviation behavior of configurations. Specifically, typical behavior is characterized by an ergodic theorem (also known as the Shannon-McMillan-Breiman theorem or the law of large numbers, depending on the context), while large deviation behavior is addressed through a Cram\'er’s theorem. This focus is influenced by \cite{Berger1990} and its successors, where the researchers examined the entropy of the stationary random fields on both rooted and unrooted $d$-trees. Subsequent works by the same authors \cite{ye1996ergodic,ye1998information}, building upon Pemantle’s combinatorial approach \cite{pemantle1992automorphism}, established a Shannon-McMillan-Breiman theorem with convergence in probability for ergodic random fields on homogeneous trees. Further developments can be found in \cite{Huang2019,shi2022generalized,Yang2003,Yang2000,Weigu-2007} and references therein, with noteworthy contributions regarding deviation inequalities provided by \cite{Peng2010,shi2022class,Yang2012}.

Broadly speaking, the previously mentioned studies center on the asymptotic behavior of the following empirical average over the initial $n$-subtrees $\lattice{n}:=\cup_{i=0}^{n} \Sigma^i$: 
\begin{equation} \label{eq:pair_empirical_averages}
    Y_n = \frac{1}{\norm{\lattice{n}}} \sum_{g \in \lattice{n} \setminus \{\epsilon\}} \log \obsfunc_{X_{\tilde{g}}, X_g},
\end{equation}
where $\obsfunc \in \mathbb{R}_{>0}^{\alphabet \times \alphabet}$ is the \emph{weight} matrix and $\norm{\cdot}$ denotes the cardinality of a set, with $\norm{\Sigma^n} = d^n$ and $\norm{\lattice{n}} = \sum_{i=0}^n {\norm{\Sigma^i}} = \frac{d^{n+1}-1}{d-1}$, in particular. As it turns out, the empirical averages $Y_n$ exhibit an asymptotically periodic structure whose period is closely linked to that of the initial state. Recall that the \emph{period} of a state $a$ in a Markov chain is defined as $p = \gcd\{n \in \mathbb{N}: (\transmat^n)_{a,a} > 0\}$, with $p = \infty$ if no such $n$ exists. Recall also that a nonnegative matrix $A \in \mathbb{R}_{\ge 0}^{\alphabet \times \alphabet}$ is \emph{irreducible} if for all $a, b \in \alphabet$ there exists $n \in \mathbb{N}$ such that $(A^n)_{a,b} > 0$, and is said to be \emph{primitive} if the condition holds for all large $n$. A Markov chain is said to be \emph{irreducible} (respectively, \emph{ergodic}) if it has an irreducible (respectively, primitive) transition matrix. As readily checked facts, all states in an irreducible Markov chain share a common period, which is referred to as the period of the Markov chain, and the period of an ergodic Markov chain is $1$. In summary, the works listed in the previous paragraph successfully recover several classical theorems for ergodic Markov chains, but leave open the question of how the results extend explicitly to more general systems, in which the behavior of $Y_n$ differs significantly from its conventional counterpart. Particularly, unlike classical Markov chains, almost sure convergence for \eqref{eq:pair_empirical_averages} cannot be guaranteed even for irreducible Markov chains (see Example \ref{ex:examples}). In light of this, the conditional averages $(Y_{p n+j}|X_{\epsilon} = a_0)$ ($n \in \mathbb{N}$) with initial state $a_0$ of finite period $p$ will be studied in place of $Y_n$ throughout the work. Agreeing with this philosophy, our first result is an analog of Cram\'er's theorem, which is presented after the introduction of necessary definitions and notations below.

Recall that given a positive real sequence $(\varepsilon_n)_{n \in \mathbb{N}}$ tending to zero and a lower semicontinuous function $I: \mathbb{R} \to [0, \infty]$, a sequence of real-valued random variables $(X_n)_{n \in \mathbb{N}}$ is said to satisfy the \emph{large deviation principle} with speed $(\varepsilon_n)_{n \in \mathbb{N}}$ and rate $I$ if for all Borel set $S \subset \mathbb{R}$,
\begin{equation*}
    \sup_{\alpha \in \mathring{S}} - I(\alpha) \le \liminf_{n \to \infty} \varepsilon_n \log \mathbb{P}(X_n \in S) \le \limsup_{n \to \infty} \varepsilon_n \log \mathbb{P}(X_n \in S) \le  \sup_{\alpha \in \overline{S}} - I(\alpha).
\end{equation*}
For matrices $A, B \in \mathbb{R}^{\alphabet \times \alphabet}$ (and similarly $\mathbb{R}^{\alphabet}$) and $r > 0$, we denote by $A \odot B :=(A_{a,b} B_{a,b})_{a,b \in \alphabet}$ the Hadamard product of $A$ and $B$, and by $A^{\odot r} := (A_{a,b}^r)_{a,b \in \alphabet}$ the element-wise power of $A$ whenever well-defined. With these notations, we define the operator $\Psi_{A,r}: \mathbb{R}_{\ge 0}^{\alphabet} \to \mathbb{R}_{\ge 0}^{\alphabet}$ by
\begin{equation} \label{eq:def_Psi}
    \Psi_{A, r}(x) = (A x)^{\odot r}.
\end{equation} 
\begin{theorem}[Cram\'er] \label{thm:1} 
    Let $\tree$ be a rooted $d$-tree, $(X_g)_{g \in \tree}$ be a finite-state Markov chain with transition matrix $\transmat$, and $(Y_n)_{n \in \mathbb{N}}$ be the empirical average with respect to the weight $A \in \mathbb{R}_{> 0}^{\alphabet \times \alphabet}$ as in \eqref{eq:pair_empirical_averages}. If $a_0 \in \alphabet$ is a state of finite period $p$, then $(Y_{p n + j}|X_{\epsilon} = a_0)_{n \in \mathbb{N}}$ satisfy the large deviation principle with speed $(\norm{\lattice{p n + j}}^{-1})_{n \in \mathbb{N}}$ and rate 
    \[
    \Lambda^*_{j}(\alpha) = \sup_{\mu \in \mathbb{R}} [\mu \alpha - \Lambda_j(\mu)],
    \]
    where, with $\onevec \in \mathbb{R}^{\alphabet}$ denoting the all-one vector,
    \begin{equation*}
        \Lambda_j(\mu) = \lim_{n \to \infty} \frac{\log ((\Psi_{\transmat \odot \obsfunc^{\odot \mu},d})^{pn+j}(\onevec))_{a_0}}{\norm{\lattice{pn+j}}}.
    \end{equation*}
\end{theorem}
\noindent A few comments about the theorem should be mentioned. First of all, we should emphasize that assuming the initial state has finite period involves no loss generality in the characterization of $Y_n$. Precisely, we observe that the random tree 
\[
\tilde{\tree} = \{g \in \tree: X_h \text{ has infinite period for all } h \le g, h \ne g\}
\]
is almost surely a subtree of $\lattice{\norm{\alphabet}}$, and its leaf vertices $\partial \tilde{\tree} = \{h \in \tilde{\tree}: h i \notin \tilde{\tree}, \forall i \in \Sigma\}$ are associated only with states of finite period. This allows us to pass Theorem \ref{thm:1} to $(Y_n|X_g, g \in \partial \tilde{\tree})$. Second, a key challenge in proving Theorem \ref{thm:1} lies in the uncertainty of differentiability of $\Lambda_j$, and hence more sophisticated arguments, such as G\"artner-Ellis theorem, do not apply. This forces us to resort to a combinatorial reasoning. Finally, results on large deviations of tree-indexed Markov chains are also considered in \cite{DEMBO2005} for trees generated by a Galton-Watson process with a positive probability of producing zero offspring, which differs from our setting. 

Despite its tediousness, an advantage of our combinatorial approach lies in that it almost immediately gives the following version of ergodic theorem for irreducible Markov chains.
\begin{theorem} \label{thm:2}
    Under the assumptions of Theorem \ref{thm:1}, if the Markov chain is irreducible of period $p$, then for any state $a_0 \in \alphabet$,
    \begin{align*}
        & \lim_{n \to \infty} (Y_{p n + j} | X_{\epsilon} = a_0) = \sum_{i=0}^{p-1} \frac{d^{-i}}{\sum_{\ell=0}^{p-1} d^{-\ell}} \sum_{a,b \in \alphabet} \initvec^{(i+1-j)}_{a} \transmat_{a,b} \log \obsfunc_{a,b} \qquad \mathbb{P}\text{-a.s.},
    \end{align*}
    where $(\initvec^{(i)})_{i \in \mathbb{Z}}$ are the unique left probability eigenvectors of $\transmat^p$ satisfying $\initvec^{(i)} = \initvec^{(i+1)} \transmat$ and $\initvec^{(i)}_{a_0} > 0$ if and only if $i \equiv 0 ~(\bmod\ p)$.
\end{theorem}
\noindent This confirms the cyclic structure of convergence of $Y_n$ and complements the state-of-the-art results by extending the limit theorems from primitive transition matrices to irreducible ones. This periodic pointwise convergence will be utilized in the subsequent sections.

The second part of the article is motivated by another intriguing aspect of the tree-indexed Markov chains: the ``size'' of the outcome space of labeled trees $\tshift[][\incmat]$ (recall \eqref{eq:tree-shifts}). This outcome space, also known as a Markov hom tree-shift, represents a special family of symbolic systems introduced by Aubrun and B\'eal \cite{Aubrun2009} that can be viewed as a generalization of the one-sided Markov subshifts, considering that the latter could be defined using the former with the underlying tree $\tree = \mathbb{Z}_+$. Building  upon this concept, Petersen and Salama \cite{Petersen2018b,Petersen2020} introduced the topological entropy for such spaces as an analog to topological entropy for traditional shift spaces to investigate their complexity. Results therein were later generalized to asymptotic pressure for a broad class of systems \cite{Petersen2021}. Notably, the topological entropy could also be interpreted as the box-counting dimension under a proper metric (see \eqref{eq:tree-shift_metric}). Motivated by this observation, the present paper aims to investigate the Hausdorff dimension under the specified metric and relate such a quantity to the eigenvalues of certain nonlinear transfer operators: Drawing inspiration from the classical thermodynamic formalisms, the candidates of transfer operator $\mathcal{L}_{A,r}: \mathbb{R}_{\ge 0}^\alphabet \to \mathbb{R}_{\ge 0}^\alphabet$ for $r=(r_0,r_1,\cdots,r_{p-1}) \in \mathbb{R}_{>0}^{p}$ and $A \in \mathbb{R}_{\ge 0}^{\alphabet \times \alphabet}$ are defined as
\begin{equation} \label{eq:transfer_operator}
    \mathcal{L}_{A,r}(x) = \Psi_{A,r_{p-1}} \circ \Psi_{A,r_{p-2}} \circ \cdots \circ \Psi_{A,r_0}(x),
\end{equation}
with each $\Psi_{A,r_i}(x)$ defined in \eqref{eq:def_Psi}. The Hausdorff dimension is then related to the eigenvalues of the nonlinear operator, as is seen in the following theorem. 
\begin{theorem} \label{thm:Hausdorff_dim}
    Let $\tshift[][\incmat]$ be a Markov hom tree-shift with an irreducible incidence matrix $\incmat$ of period $p$. Then,
    \begin{equation} \label{eq:Hausdorff_dimension}
        \dim_{H} \tshift[][\incmat] = \min_{r \in \mathcal{R}_{p,d}} \left(\sum_{\ell=0}^{p-1} \prod_{i=0}^{\ell} r_{i}^{-1}\right)^{-1} \cdot \log \rho_*(\mathcal{L}_{\incmat,r}),
    \end{equation}
    where 
    \begin{gather*}
        \mathcal{R}_{p,d} = \left\{r \in (0,d]^p: \prod_{i=0}^{p-1} r_i = 1\right\}, \\
        \rho_*(\mathcal{L}_{\incmat,r}) = \inf \{\alpha > 0: \mathcal{L}_{\incmat,r}(u) = \alpha \cdot u \in \mathbb{R}^{\alphabet} \setminus \{0\}\}.
    \end{gather*}
    In particular, if $\incmat$ is a primitive matrix, then $p=1$, $\mathcal{L}_{\incmat,r} = \incmat$ for $r \in \mathcal{R}_{1,d}=\{1\}$, and $\dim_{H} \tshift[][\incmat] = \log \rho_*(\incmat) = \log \rho(\incmat)$, where $\rho(\incmat)$ is the spectral radius of $\incmat$.
\end{theorem}
\noindent It is noteworthy that the theorem formally generalizes the classical thermodynamic formalisms. Specifically, when $d =1$, the set $\mathcal{R}_{p,d}$ simplifies to $ = \{(1,1, \cdots,1)\}$, leading formally to the conclusion that $\dim_H \tshift[][\incmat] = p^{-1} \log \rho(\incmat^p) = \log \rho(\incmat)$, a known result for irreducible Markov subshifts. 

At this stage, two main concerns arise regarding Theorem \ref{thm:Hausdorff_dim}: the existence of nonnegative eigenvectors and the attainment of $\min_{r \in \mathcal{R}_{p,d}}$. To address the former issue, the proof of the theorem is deferred to Section \ref{sec:Perron-Frobenius}, where a brief overview of nonlinear Perron-Frobenius theory is provided to establish the required existence. The latter issue is discussed in Lemma \ref{lem:periodic_maximality}, which forms the cornerstone in the proof of Theorem \ref{thm:Hausdorff_dim}. Notably, the proposed formula of the Hausdorff dimension can be further adapted to serve as an upper bound for any Markov hom tree-shifts, leading to the following corollary.
\begin{corollary} \label{cor:upper_bound_spectral_radius}
    Let $\tshift[][\incmat]$ be a Markov hom tree-shift. Then, $\dim_{H} \tshift[][\incmat] \le \log \rho(\incmat)$. In particular, for any irreducible $\incmat$, the equality holds if and only if $\incmat$ has uniform row sums.
\end{corollary}
To provide an overview of our proof strategy, it is worth mentioning that the upper bound of the Hausdorff dimension is derived via constructing efficient covers of $\tshift[][\incmat]$, while the lower bound is obtained by applying the classical mass distribution principle to the ``optimal'' Markov measure. Hence, the main contributions of this work lie not only in providing a formula for Hausdorff dimension but also in establishing an analog of the variational principle for the tree-shifts.

The organization of the paper is as follows. In Section \ref{sec:preliminaries}, we introduce preliminary definitions and establish the general conventions. Section \ref{sec:large_deviation} is devoted to the proofs of Theorems \ref{thm:1} and \ref{thm:2}. Section \ref{sec:Hausdorff_dimension_irr} presents the proposed variational formula for Hausdorff dimensions (Theorem \ref{thm:Hausdorff_dim}), which involves a recapitulation of the nonlinear Perron-Frobenius theory (Section \ref{sec:Perron-Frobenius}), followed by necessary lemmas on duality (Section \ref{sec:some_lemmas}) and lower and upper bounds for Hausdorff dimensions (Sections \ref{sec:lower_bound} and \ref{sec:upper_bound}). Section \ref{sec:general_upper_bound} concludes with a discussion of a general upper bound for the Hausdorff dimension and an illustrative example validating our dimension formula. Finally, the paper concludes with a brief discussion in Section \ref{sec:discussion}.

\section{Preliminaries} \label{sec:preliminaries}
\begin{table}[]
    \centering
    \begin{tabular}{c l}
        \hline
        $\Sigma$ & generating set of the semigroup \\
        $\tree$ & free semigroup generated by $\Sigma$ \\
        $\level{n}$ & alias of $\Sigma^n$ \\
        $\lattice{n}$ & $\cup_{i=0}^{n} \level{i}$ \\
        $\alphabet$ & state space \\
        $\incmat$ & incidence matrix \\
        $\tshift[][\incmat]$ & Markov hom tree-shift associated with $\incmat$ \\
        $[k]$ & $\{0,1,\cdots,k-1\}$ \\
        $\tilde{g}$ & parent of $g \in \tree$ \\
        $a_0$ & label satisfying \eqref{eq:tree_assumption} \\
        $p$ & period of state $a_0$ \\
        $\mathcal{P}(\incmat)$ & $\{\alphabet_j: j \in [p]\}$ \\
        $\alphpv_{\alphabet}$, $\alphpv$ & set of probability vectors indexed by $\alphabet$ \\
        $\alphsm_{\alphabet}$, $\alphsm$ & set of stochastic matrix indexed by $\alphabet$ \\
        $\pv$, $\tv$ & vectors in $\Gamma_{\alphabet}$ \\
        $\transmat$, $\pm$, $\tm$ & matrices in $\Upsilon_{\alphabet}$ \\
        $\dv{n}(t)$, $\dv{n}[m](t)$ & distribution vectors associated with $t \in \tshift[][\incmat]$ \\
        $\trm{n}(t)$, $\trm{n}[m](t)$ & transition matrices associated with $t \in \tshift[][\incmat]$ \\
        $\mathsf{d}_{v}$, $\mathsf{d}_{V}$, $\mathsf{d}_{v,V}$ & variational distances \\
        $\dvset{n}[m]$,$\dvset{n}$ & set of feasible distributions by $\tshift[][\incmat]$ \\
        $\trmset{n}[m]$,$\trmset{n}$ & set of feasible transitions by $\tshift[][\incmat]$ \\
        $\proddom*{n}[m]$,$\proddom*{n}$ & set of collectively feasible distributions and transitions by $\tshift[][\incmat]$ \\
        $\block{n}[m]$, $\block{n}$ & feasible patterns of $\tshift[][\incmat]$ \\
        $\block{n}[m](\tv,\tm)$, $\block{n}(\tv,\tm)$ & feasible patterns of $\tshift[][\incmat]$ with distributions $\tv$ and transitions $\tm$ \\
        $(\pv,\pm) = (\tv*,\tm*)$ & reversed sequence of $(\tv,\tm)$ \\
        $Z$, $Z_j$ & feasible set of extended reversed sequences \\
        $\mathcal{C}_j$ & vectors supported on $\alphabet_j$ \\
        $\DKL(\pm||A)$ & relative entropy of $\pm$ and $A$ \\
        $A \odot B$ & Hadamard product of $A$ and $B$ \\
        $A^{\odot r}$ & matrix or vector $A$ raised to power $r$ \\
        $(X_g)_{g \in \tree}$ & Markov chain indexed by $\tree$ \\
        \hline
    \end{tabular}
    \caption{Summary of notation}
    \label{tab:notation_summary}
\end{table}
We let $\tree$, $\Sigma$, $\lattice{n}$, and so on, be as defined and use the alias $\level{n} = \Sigma^n$ to avoid possible confusion. We will consistently denote the transition matrix as $\transmat$ and its incidence matrix as $\incmat$. To simplify our discussion, two assumptions will be made throughout unless otherwise mentioned. The first one is rather fundamental, which states as follows:
\begin{equation}
\label{eq:nontrivial_assumption}
  \tag{A0}
  \parbox{\dimexpr\linewidth-3.5em}{\centering
    \strut
    $\sum_{b} \incmat_{a,b} > 0$ for all $a \in \alphabet$.
    \strut
  }
\end{equation}
By assuming so, we avoid ``nonessential'' states in our discussion, as any $a$ failing the criterion would appear in neither $\mathcal{T}_{\incmat}$ nor the Markov chain, and can be easily excluded by restricting to a subset $\alphabet' = \alphabet \setminus \{a\}$ of the state space without actually altering the systems. The second assumption owes to our exclusive attention to states of finite period as in Theorem \ref{thm:1}, and it requires the following:
\begin{equation}
\label{eq:tree_assumption}
  \tag{A1}
  \centering
  \parbox{\dimexpr\linewidth-3.4em}{%
    \strut
    There exists $a_0 \in \alphabet$ such that every $a \in \alphabet$ admits $n \in \mathbb{N}$ satisfying $(\incmat^{n})_{a_0,a} > 0$.
    \strut
  }
\end{equation}
Though such $a_0$ might not be unique, the choice will not affect our discussion and will henceforth be fixed. Notably, this assumption becomes transparent if the incidence matrix $\incmat$ is irreducible, as it is automatically satisfied under the circumstances. It plays a role only when $\incmat$ is not irreducible, in which case states are pruned if they would not be seen given the initial state $a_0$, simplifying discussion of Theorem \ref{thm:1}. We should also mention that such requirement will only be dropped in Section \ref{sec:general_upper_bound} when deriving an upper bound of Hausdorff dimensions for general Markov hom tree-shift. By writing $[k]=\{0,1,\cdots,k-1\}$ for $k \in \mathbb{N}$, the assumption \eqref{eq:tree_assumption} naturally induces a collection $\mathcal{P}(\incmat) = \{\alphabet_j: j \in [p]\}$ of subsets of $\alphabet$ such that
\begin{equation} \label{eq:label_decomposition}
    a \in \alphabet_j \text{ if and only if } (\incmat^n)_{a_0, a} > 0 \text{ for some } n \equiv j~(\bmod\ p),
\end{equation}
where we set $\alphabet_{i+p} = \alphabet_{i}$ for all $i \in \mathbb{Z}$. Although technical, $\mathcal{P}(\incmat)$ is essential for characterizing the cyclic structure of convergence in Theorem \ref{thm:1} and \ref{thm:2}. In particular, $\mathcal{P}(\incmat)$ forms a partition of $\alphabet$ if $\incmat$ is irreducible.

Some frequently used symbols are borrowed from the work of the first and the third authors \cite{ban2023topological}. We let $\Gamma_{\alphabet}$ be the set of all probability vectors indexed by $\alphabet$, and $\Upsilon_{\alphabet}$ be the set of stochastic matrices acting on $\Gamma_{\alphabet}$; explicitly, this means that given any $\pv \in \Gamma_{\alphabet}$ and any $\pm \in \Upsilon_{\alphabet}$, $\pv$, as a row vector, can always be multiplied by $\pm$ from the right so that their product $\pv \pm$ is still a probability vector. For conciseness, $\alphabet$ is suppressed from $\Gamma_{\alphabet}$ and $\Upsilon_{\alphabet}$, unless other index sets are involved. We define the \emph{variational distances} $\mathsf{d}_{v}$ and $\mathsf{d}_{V}$ on $\alphpv$ and $\alphsm$, respectively, as
\[
\mathsf{d}_{v}(\pv,\tv) = \frac{1}{2} \sum_{a \in S} \left|\pv_a-\tv_a\right| \text{ and } \mathsf{d}_{V}(\pm,\tm)=\max_{a \in \alphabet} \mathsf{d}_{v} ((\pm_{a, b})_{b \in \alphabet},(\tm_{a, b})_{b \in \alphabet}).
\]
Associated with the metrics is a sup metric $\mathsf{d}_{v,V}((\pv,\pm),(\tv,\tm))$ on the product space $\alphpv \times \alphsm$ defined as
\[
\mathsf{d}_{v,V}((\pv,\pm),(\tv,\tm))=\max\{\mathsf{d}_{v}(\pv,\tv), \mathsf{d}_{V}(\pm,\tm)\}.
\]
For a labeled tree $t \in \alphabet^{\tree}$, the \emph{$n$-th level distribution} of $t$ is denoted as
\[
\dv{n}(t)=\left(\frac{\norm{\{g \in \level{n}: t_g = a\}}}{\norm{\level{n}}}\right)_{a \in \alphabet} \in \alphpv,
\]
In addition, the \emph{$n$-th level transition} of $t$ is written as
\[
\trm{n}(t)_{a, b}=\begin{cases}
    \left(\frac{\norm{\{g \in \level{n+1}: t_{\tilde{g}} = a, t_g = b\}}}{\norm{\{g \in \level{n+1}: t_{\tilde{g}} = a\}}} \right) & \text{if } \norm{\{g \in \level{n+1}: t_{\tilde{g}} = a\}} > 0, \\
    \frac{\incmat_{a, b}}{\sum_{c \in \alphabet} \incmat_{c, b}} & \text{otherwise}.
\end{cases}
\]
For simplicity, we write the distributions of $t$ from level $n$ to level $m$ ($n \le m$) by $\dv{n}[m](t)=(\dv{n}(t),\cdots, \dv{m}(t))$, associated with which we put 
\[
\dvset{n}[m] =\left\{\dv{n}[m](t): t \in \tshift[][\incmat]\right\}
\]
to be the set of set of all admissible distributions. Likewise, the transitions of $t$ from level $n$ to level $m$ are denoted by $\trm{n}[m-1](t)=(\trm{n}(t),\cdots, \trm{m-1}(t))$ and 
\[
\trmset{n}[m]=\left\{\trm{n}[m-1](t): t \in \tshift[][\incmat]\right\}
\]
is the set of all admissible transitions. Also, let 
\[
\proddom*{n}[m]=\{(\dv{n}[m](t),\trm{n}[m-1](t)): t \in \tshift[][\incmat] \}
\]
be the collective admissible sets of distributions and transitions. In a similar fashion, we may extend the definition of a subtree $\lattice{n}$ to a subgraph consisting of vertices from level $n$ to level $m$:
\[
\lattice{n}[m] = \cup_{i=n}^m \level{i},
\]
so the set of admissible blocks by prescribed distributions $\tv=(\tv[0],\cdots,\tv[m-n])$ and transitions $\tm=(\tm[0],\cdots,\tm[m-n-1])$ is defined as
\begin{equation*} \label{eq:block_pattern}
\block{n}[m](\tv,\tm)=\{t|_{\lattice{n}[m]} \in \alphabet^{\lattice{n}[m]}: t \in \tshift[][\incmat], \dv{m}[n](t) = \tv, \trm{n}[m-1](t) = \tm\}.
\end{equation*}
In particular, the set of all admissible blocks $\block{n}[m]$ is written as the union of $\block{n}[m](\tv,\tm)$:
\[
\block{n}[m] = \cup_{(\tv,\tm) \in \proddom*{n}[m]} \block{n}[m](\tv,\tm).
\]
For simplicity, ``$n:$'' is suppressed in the notation of $\dvset{n}[m]$, $\trmset{n}[m]$, $\proddom*{n}[m]$, and $\block{n}[m]$ if $n=0$. Notably, for $(\tv,\tm) \in \proddom*{n}[m]$, it is necessary that (a) $\tv[i+1] = \tv[i] \tm[i]$ , and (b) $\tm[i]_{a,b}=0$ if $\incmat_{a,b}=0$ for every $a, b \in \alphabet$ and all $0 \le i < m-n$. For brevity, We write $\tm[i] \prec \incmat$ whenever (b) is satisfied and write $\tm[i] \sim \incmat$ when both $\tm[i] \prec \incmat$ and $\incmat \prec \tm[i]$. For convenience's sake, we denote by $(\pv,\pm) = (\tv*,\tm*)$ the reversed sequence of $(\tv,\tm)$, (i.e., $(\pv[0], \cdots, \pv[n-m])=(\tv[n-m], \cdots, \tv[0])$ and $(\pm[0], \cdots, \pm[n-m-1])=(\tm[n-m-1], \cdots, \tm[0])$) and consider the following set of all extended reversed sequences:
\begin{equation} \label{eq:support}
    Z :=\{(\mathsf{p},\mathsf{P}) \in \Gamma^{\mathbb{Z}_+} \times \Upsilon^{\mathbb{Z}_+}: \mathsf{p}^{(i)}=\mathsf{p}^{(i+1)} \pm[i] \text{ and } \pm[i] \prec \incmat \text{ for all } i \in \mathbb{Z}_+\},
\end{equation}
together with which we assign 
\begin{equation} \label{eq:support_C}
    \mathcal{C}_j = \{x \in \mathbb{R}_{\ge 0}^\alphabet: x_a = 0 \text{ if } a \notin \alphabet_j\}
\end{equation}
and
\begin{equation} \label{eq:support_j}
    Z_j = \{(\pv,\pm) \in Z: \pv[i] \in \mathcal{C}_{j-i} \text{ for all } i \in \mathbb{Z}_+\}.
\end{equation}
Additionally, for any sequence of probability vectors (and similarly for stochastic matrices), say $\pv \in \alphpv^{\mathbb{Z}_+}$, we write 
\[
\pv[i:j] = (\pv[i],\pv[i+1],\cdots,\pv[j]), \quad 0 \le i \le j \le \infty. 
\]
Finally, the convention of matrices and vectors is manifested as follows. Matrices are typed in uppercase while vectors are in lowercase, and stochastic matrices and probability vectors are in sans serif font, such as $\pm$, $\tm$, $\pv$, and $\tv$. In particular, $\{\stdvec{a}: a \in \alphabet\}$ is the standard basis of $\mathbb{R}^{\alphabet}$. Except for probability vectors in $\alphpv$, every vector in the article is by default a column vector. In addition, for any $\pm \in \alphsm$ and any nonnegative matrix $A \in \mathbb{R}_{\ge 0}^{\alphabet \times \alphabet}$ satisfying $\pm \prec A$, we define their ``Kullback-Leibler divergence'' as
\[
\DKL(\pm || A)_b:=\sum_{b \in \alphabet} \pm_{a, b} \log\left(\frac{\pm_{a, b}}{A_{a, b}}\right),
\]
where $0 \log \frac{0}{A_{a,b}}$ is interpreted as $0$ for $A_{a,b} \ge 0$. Also, when functions, such as division $/$, exponential function $e$, and logarithm $\log$, are acting on matrices or vectors, the actions are by default entrywise unless stated otherwise.
The notations are summarized in Table \ref{tab:notation_summary}.

\section{Large deviations of Markov chains on $d$-trees} \label{sec:large_deviation}
The aim of this section is to carry out the proof of Theorem \ref{thm:1}, which states the large deviation principle for the conditional sample mean. 

To illustrate our proof strategy, we present a heuristic proof of the following version of Cram\'er's theorem (see, e.g., \cite[Theorem 2.1.24]{Dembo2009-qs} for a similar version).
\begin{theorem}[Cram\'er's theorem] \label{thm:Cramer}
    Let $(X_i)_{i \in \mathbb{N}}$ be a sequence of i.i.d.~random variables taking values on a finite subset of $\mathbb{R}$, say $\{a_1, a_2, \cdots, a_k\}$. Then, 
    \[
    \frac{1}{n} \log \mathbb{P}\left(\frac{1}{n} \sum_{i=1}^{n} X_i \ge \alpha\right) \to \sup_{\mu \in \mathbb {R}} \left[\mu \alpha -\Lambda (\mu)\right] \text{ for } \alpha \ge \mathbb{E} [X_1],
    \]
    where
    \[
    \Lambda(\mu)=\log \mathbb{E} [e^{\mu X_{1}}] = \log \sum_{j=1}^{k} p_j \cdot e^{\mu a_j} \text{ with } (p_j)_{j=1}^k = (\mathbb{P}(X_1 = a_j))_{j=1}^k.
    \]
\end{theorem}
\noindent By virtue of Stirling's approximation, one can reformulate the probability in terms of an optimization problem:
\begin{align*}
    & \frac{1}{n} \log \mathbb{P}\left(\frac{1}{n} \sum_{i=1}^{n} X_i \ge \alpha\right) \\
    = & \sup_{\beta \ge \alpha} \left\{\frac{1}{n} \log \mathbb{P}\left(\sum_{j=1}^{k} t_j a_j  = \beta\right): t_j = \frac{1}{n} \sum_{i=1}^{n} \chi_{a_j}(X_i) \right\} + O(n^{-1} \log n) \\
    = & \sup_{\beta \ge \alpha} \left\{ \sum_{j=1}^{k} t_j \log \frac{p_j}{t_j}: \sum_{j=1}^{k} t_j a_j = \beta \right\} + o(1),
\end{align*}
where $\chi_a$ denotes the characteristic function of the state $a$ and the first equality holds since the set
\[
L_n = \left\{(t_j)_{j=1}^{k}: t_j = \frac{1}{n} \sum_{i=1}^{n} \chi_{a_j}(X_i)\right\}
\]
has its cardinality bounded from above by $\binom{n}{k-1}$, and the second from the fact that $L_n$ is asymptotically dense in the simplex of probability vectors. The theorem is ultimately proved by deriving the rate function $\sup_{\mu \in \mathbb{R}} \left[\mu \alpha - \Lambda(\mu)\right]$, as a dual problem, utilizing standard techniques from optimization theory. Inspired by this argument, the authors analogously reproduced the necessary approximations and estimations in \cite{ban2023topological}, based upon which the method of types (see, e.g., \cite[Chapter 2.1.1]{Dembo2009-qs}) can be adapted for Theorem \ref{thm:1}.

To prove Theorem \ref{thm:1}, our starting point would be an observation asserting that the topological space $\alphpv^{\mathbb{Z}_+} \times \alphsm^{\mathbb{Z}_+}$ is metrizable, and that a natural choice of the metric would be induced by the variational distance on $\alphpv^{\mathbb{Z}_+} \times \alphsm^{\mathbb{Z}_+}$, which is defined as
\begin{equation} \label{eq:product_metric}
    \mathsf{d}^{\infty}_{v,V}((\mathsf{p},\mathsf{P}),(\mathsf{q},\mathsf{Q})):=\sum_{i=0}^{\infty} \frac{d-1}{d^{i+1}} \cdot \mathsf{d}_{v,V}((\mathsf{p}^{(i)},\pm[i]),(\mathsf{q}^{(i)},\mathsf{Q}^{(i)})) \qquad (d \in \mathbb{N} \setminus \{1\}).
\end{equation}
We note that the product topology is compatible with this metric, rendering $\alphpv^{\mathbb{Z}_+} \times \alphsm^{\mathbb{Z}_+}$ a compact metric space. For the scope of this article, the following two families of functions on $Z$ are of interest. For $0 \le n \le m \le \infty$ and $A \in \mathbb{R}_{\ge 0}^{\alphabet \times \alphabet}$ satisfying $\incmat \prec A$, define $\mathcal{F}_{n:m}, \mathcal{G}_{n:m}: Z \to \mathbb{R}$ (with parameter $A$) by
\begin{gather*}
    \mathcal{F}_{n:m}(\mathsf{p},\mathsf{P};A) = -\sum_{i=n}^{m} \frac{d-1}{d^{i-n+1}} \mathsf{p}^{(i+1)} \DKL(\pm[i] || A), \\
    \mathcal{G}_{n:m}(\mathsf{p},\mathsf{P};A) = \sum_{i=n}^{m} \frac{d-1}{d^{i-n+1}} \mathsf{p}^{(i+1)} (\pm[i] \odot \log \obsfunc) \onevec,
\end{gather*}
where, as previously, $A \odot B$ denotes the Hadamard product of matrices $A$ and $B$, the logarithm is applied entrywise, and $0 \log 0 = 0$. For conciseness, we suppress ``$n:$'' when $n = 0$ and ``$m$'' when $m = \infty$. The functions $\mathcal{F}$ and $\mathcal{G}$ turn out to be continuous, as shown in the following lemma.
\begin{lemma} \label{lem:1}
    Suppose $A \in \mathbb{R}_{\ge 0}^{\alphabet \times \alphabet}$ satisfying $\incmat \prec A$. The functions $\mathcal{F}(\cdot;A)$, $\mathcal{G}(\cdot;A):Z \to \mathbb{R}$ are uniformly convergent and thus continuous.
\end{lemma}
\begin{proof}
    Indeed, it is because the functions $(\pv,\pm) \mapsto \mathsf{p}^{(i+1)} \DKL(\pm || A)$ and $(\pv,\pm) \mapsto \mathsf{p}^{(i+1)} (\pm[i] \odot \log A) \onevec$ are continuous and uniformly bounded over $Z$.
\end{proof}
\noindent The continuity and compactness are crucial to the rest of our discussion. In particular, they guarantee that for $\alpha \in \mathbb{R}$ and all $n \in \mathbb{N} \cup \{\infty\}$, the domains
\begin{align*}
    Z^{(n)}(\alpha;\obsfunc) = \{(\pv,\pm) \in Z: \mathcal{G}_n(\pv,\pm; \obsfunc) = \alpha\} \text{ and } Z_j^{(n)}(\alpha;\obsfunc) = Z^{(n)}(\alpha;\obsfunc) \cap Z_j
\end{align*}
are compact, of which we omit the superscript ``$(n)$'' if $n = \infty$.

Our next step is to discuss combinatorial approximations of the empirical distribution of patterns, which are manifested in the following series of lemmas. 
\begin{lemma}[Proposition 2.2, \cite{ban2023topological}] \label{lem:dist_set_growth_rate}
    For any $n \le m$, 
    \begin{gather*} 
        1 \le \norm{\dvset{n}[m]} \le \prod_{i=n}^{m-1} \left(\norm{\level{i}}+1\right)^{\norm{\alphabet}} \le \left(\frac{\norm{\lattice{n}[m]}}{m-n+1} + 1\right)^{(m-n+1) \cdot \norm{\alphabet}}, \\
        1 \le \norm{\trmset{n}[m]} \le \prod_{i=n}^{m-1} \left(\norm{\level{i}}+1\right)^{\norm{\alphabet} (\norm{\alphabet}+1)} \le \left(\frac{\norm{\lattice{n}[m]}}{m-n+1} + 1\right)^{2 (m-n+1) \cdot \norm{\alphabet}^2},
    \end{gather*}
    where $\lattice{n}[m] = \cup_{i=n}^m \level{i}$ has cardinality $\sum_{i=n}^m = d^i$.
\end{lemma}
\noindent The lemma indicates that the growth of admissible distributions and transitions is subexponential with respect to $\lattice{n}[m]$, which is markedly smaller than the exponential growth rate of the cardinality of the set $\block{n}[m]$. The next two lemmas, on the other hand, demonstrate that the set $\proddom*{n}[m]$ is, in some sense, asymptotically dense in $Z$, and Stirling's approximations can be applied as in our heuristic proof of Cram\'er's theorem. Indeed, if we define
\begin{multline*}
    \proddom{k}'=\left\{(\pv,\pm) \in \left(\alphpv^{k} \times \{\stdvec{a}\}_{a \in \alphabet}\right) \times \alphsm^{k}: \pm[i] \prec \incmat, \pv[i+1] = \pv[i] \pm[i],   \vphantom{\tfrac{\incmat_{a,b}}{{\textstyle\sum}_{c \in \alphabet} \incmat_{c,b}}}\right. \\
    \left.\pm[i]_{a,b} = \tfrac{\incmat_{a,b}}{{\textstyle\sum}_{c \in \alphabet} \incmat_{c,b}} \text{ if } \pv[i+1]_b=0 \text{ for all } i \in [k]\right\},
\end{multline*}
we have the following lemmas.
\begin{lemma}[{\cite[Proposition 3.2]{ban2023topological}}] \label{lem:dense_pattern}
    Let $\mathtt{d}^k_{v,V}$ be any product metric on the space $\alphpv^{k+1} \times \alphsm^k$, say for example
    \[
    \mathtt{d}^k_{v,V}((\pv,\pm),(\tv,\tm)):=\max \left\{\max_{0 \le i \le k}\mathtt{d}_{v}(\pv[i],\tv[i]), \max_{0 \le i < k} \mathtt{d}_{V}(\pm[i],\tm[i]) \right\}.
    \]
    Then, $\lim_{n \to \infty} \sup_{(\pv,\pm) \in \proddom{k}'} \mathtt{d}^k_{v,V}(\proddom*{n}[n+k], (\pv*,\pm*)) = 0$.
\end{lemma}
\noindent Notably, given any $(\pv,\pm) \in Z$, the values of $\mathcal{F}_n(\pv,\pm;A)$ and $\mathcal{G}_n(\pv,\pm;A)$ are completely independent of $\pm[i]_{a,b}$ whenever $\pv[i+1]_a = 0$, which renders the corresponding condition in the definition of $\proddom{k}'$ transparent.
\begin{lemma} \label{lem:uniform_equiconvergence}
    For any $(\tv, \tm) \in \proddom*{n}$ and any $(\pv,\pm) \in Z$ satisfying $(\pv^{(0:n)},\pm^{(0:n-1)}) = (\tv*,\tm*)$,
    \begin{align*}
        &\frac{1}{\norm{\lattice{n}}} \sum_{u \in \block{n}(\tv,\tm)} \sum_{g \in \lattice{n} \setminus \{\epsilon\}} \log \obsfunc_{u_{\tilde{g}},u_{g}} = (1-d^{-n-1}) \cdot \mathcal{F}_{n}(\pv,\pm;\obsfunc) + O\left(\frac{\log \norm{\lattice{n}}}{\norm{\lattice{n}}}\right),
    \end{align*}
    where the implicit constant in the big O notation is independent of $(\tv, \tm)$ and $(\pv,\pm)$.
\end{lemma}
\begin{proof}
    The equality follows from \cite[Proposition 3.1]{ban2023topological} together with the observation that $\frac{d-1}{d^{n-i+1}} / \frac{\norm{\level{i}}}{\norm{\lattice{n}}} = 1-d^{-n-1}$.
\end{proof}

Similar to the classical arguments using the method of types, the theory of convex optimization plays a key role. In the present work, we resort to Sion's minimax theorem (see, e.g., \cite{Komiya1988} for a proof). To apply it, we need to verify \textit{a priori} the concavity/convexity of the objective function, and this is the moment when the following lemma is apropos. Throughout the discussion, the convexity of a vector-valued function is defined in an entrywise sense, as stated below.
\begin{definition} \label{def:convexity}
    Let $V$ be a convex subset of a vector space. An extended vector-valued function $(f_a)_{a \in \alphabet}: V \to (\mathbb{R} \cup \{\pmold\infty\})^{\alphabet}$ is said to be \emph{convex} if each $f_a$ is convex over the domain $V$.
\end{definition}
\begin{lemma} \label{lem:convexity_backbone}
    Let $V$ be a convex set in a vector space. Suppose that
    \begin{itemize}
        \item $E: V \to \mathbb{R}_{\ge 0}^{\alphabet \times \alphabet}$ satisfies $E(x) \sim \incmat$ for all $x \in V$ and $\log E$ (entrywise logarithm) is convex,
        \item $\lambda: V \to \mathbb{R}^{\alphabet}$ is convex, and
        \item $q: V \to \mathbb{R}_{> 0}$ is affine.
    \end{itemize}
    If $E$ or $q$ is constant, then the following map is convex:
    \[
    x \mapsto q(x) \log (E(x) e^{q(x)^{-1} \lambda(x)})
    \]
    where exponentiation and logarithm are both applied entrywise.
\end{lemma}
\begin{proof}
    If $E(x) = E$ is constant, then for nonnegative $\alpha, \alpha'$ satisfying $\alpha + \alpha'=1$, we have
    \begin{equation*}
        \begin{aligned}
            & (q(\alpha x + \alpha' x')) \log \left(E e^{\frac{\lambda(\alpha x + \alpha' x')}{q (\alpha x + \alpha' x')}}\right) = \log \left(\left(E e^{\frac{\lambda(\alpha x + \alpha' x')}{\alpha q(x) + \alpha' q(x')}}\right)^{\alpha q(x) + \alpha' q(x')}\right) \\
            \le & \log \left(\left(E e^{\frac{\alpha q(x)}{\alpha q(x) + \alpha' q(x')} \frac{\lambda(x)}{q(x)} + \frac{\alpha' q(x')}{\alpha q(x) + \alpha' q(x')} \frac{\lambda(x')}{q(x')}}\right)^{\alpha q(x) + \alpha' q(x')}\right) \\
            \le & \log \left(\left(E e^{\frac{\lambda(x)}{q(x)}}\right)^{\alpha q(x)}\right) + \log \left(\left(E e^{\frac{\lambda(x')}{q(x')}}\right)^{\alpha' q(x')}\right) \\
            = & \alpha q(x) \log (E e^{q(x)^{-1} \lambda(x)}) + \alpha' q(x') \log (E e^{q(x')^{-1} \lambda(x')}),
        \end{aligned}
    \end{equation*}
    where the first inequality results from convexity and the second from H\"older's inequality. Similarly, if, without loss of generality, $q(x) = 1$ is constant, then
     \begin{equation*}
        \begin{aligned}
            & \log \left(E(\alpha x + \alpha' x') e^{\lambda(\alpha x + \alpha' x')}\right) = \log \left(e^{\log E(\alpha x + \alpha' x')} e^{\lambda(\alpha x + \alpha' x')}\right)  \\
            \le & \log \left(e^{\alpha \log E(x) + \alpha' \log E(x')} e^{\alpha \lambda(x) + \alpha' \lambda(x')}\right) \\
            \le & \alpha \log \left(E(x) e^{\lambda(x)}\right) + \alpha' \log \left(E(x') e^{\lambda(x')}\right),
        \end{aligned}
    \end{equation*}
    by H\"older's inequality.
\end{proof}

\begin{lemma} \label{lem:probability_surjectivity}
    Suppose that assumption \eqref{eq:tree_assumption} holds. Then, for all sufficiently large $n_0 \in\mathbb{N}$, $\min_{a \in \alphabet_0}(\incmat^{p n_0})_{a_0,a} > 0$. In addition, for every $\tv \in \mathcal{C}_0 \cap \alphpv$ there exists $(\pv,\pm) \in Z_0$ such that $\pv[pn_0] = \stdvec{a_0}$ and that $\pv[0] = \tv$.
\end{lemma}
\begin{proof}
    The existence of $n_0$ satisfying $\min_{a \in \alphabet_0} (\incmat^{p n_0})_{a_0,a} > 0$ follows from \eqref{eq:tree_assumption} and that $a_0$ has finite period $p$. To verify the additional properties are also satisfied, we first construct, for every $a \in \alphabet_0$, a sequence $(b^{a,i})_{i \in \mathbb{Z}_+}$ satisfying that $b^{a,0}=a$, that $b^{a,pn_0} = a_0$, and that $\incmat_{b^{a,i+1},b^{a,i}} = 1$ for all $i \in \mathbb{Z}_+$. Additionally, we assume that for every $i=0,1,\cdots,pn$, the condition $b^{a,i} = b^{a',i}$ implies $(b^{a,j})_{j \ge i} = (b^{a',j})_{j \ge i}$. To achieve this, if $b^{a,1} = b^{a',1}$ but $(b^{a,j})_{j \ge 1} \ne (b^{a',j})_{j \ge 1}$, then $(b^{a',j})_{j \ge 1}$ can be replaced by $(b^{a,j})_{j \ge 1}$, and the replacement continues until the property is satisfied for all $b^{a,1}$, $a \in \alphabet_0$. Iterating process for $i=1,\cdots,pn_0$, the property is further satisfied by $b^{a,i}$ for $i$ and all $a$. The lemma is then concluded by choosing 
    \[
    \pv[0] = \tv \text{ and } \pv[i+1]_{b} = \sum_{b^{a,i}:b^{a,i+1} = b} \pv[i]_{b^{a,i}} \quad (i \in \mathbb{Z}_+),
    \]
    and
    \[
    \pm[i]_{b,b'} = \begin{cases}
        \frac{\pv[i]_{b'}}{\pv[i+1]_{b}} & \text{if } \pv[i+1]_{b} > 0,\vspace{.5em} \\
        \frac{\incmat_{b,b'}}{\sum_{c \in \alphabet} \incmat_{b,c}} & \text{otherwise}.
    \end{cases} \quad \quad (i \in \mathbb{Z}_+)
    \]
\end{proof}

\begin{lemma} \label{lem:asymptotic_property}
    Suppose that assumption \eqref{eq:tree_assumption} holds. Let $\obsfunc \in \mathbb{R}_{\ge 0}^{\alphabet \times \alphabet}$ satisfy $\incmat \prec \obsfunc$ and define $E(\mu) = \transmat \odot \obsfunc^{\odot \mu}$ for every $\mu \in \mathbb{R}$. Then, the following hold.
    \begin{enumerate}
        \item Let $(\lambda^{(n)}(\mu))_{n \in \mathbb{Z}_+}$ be a sequence of $\mathbb{R}^{\alphabet}$-valued functions iteratively defined by
        \begin{align} \label{eq:lambda}
            \lambda^{(n+1)}(\mu) = & \frac{d-1}{d^{i+1}} \log \left(E(\mu) e^{\frac{d^{n+1}}{d-1} \lambda^{(n)}(\mu)}\right) \text{ with } \lambda^{(0)}(\mu) = 0.
        \end{align}
        Then, $\mu \mapsto \lambda^{(n)}(\mu)$ is convex for all $n \in \mathbb{Z}_+$.
        
        \item For all $a \in \alphabet$ and $n \in \mathbb{N}$, 
        \[
        \lambda^{(n)}(\mu)_a = \max_{(\pv,\pm) \in Z, \pv[n] = \stdvec{a}}\mathcal{F}_n(\pv,\pm;E(\mu)),
        \]
        \[
        \inf_{\mu \in \mathbb{R}} [-\mu \alpha + \lambda^{(n)}(\mu)_a] = \sup_{\substack{(\pv,\pm) \in Z^{(n)}(\alpha;\obsfunc), \pv[n] = \stdvec{a}}} \mathcal{F}_n(\pv,\pm; \transmat).
        \]
        
        \item Let $n_0 \in \mathbb{N}$ be a constant given in Lemma \ref{lem:probability_surjectivity}. Then, there exists an absolute constant $C > 0$ such that for all $n \in \mathbb{N}$,
        \begin{align*}
            \norm{\lambda^{(n+p n_0)}(\mu)_{a_0} - \max_{a \in \alphabet_0}\lambda^{(n)}(\mu)_{a}} \le C d^{-n} (\norm{\mu}+1).
        \end{align*}
        
        \item Let $\Lambda_j$, $\Lambda_j^*$ be as defined in Theorem \ref{thm:1}. Then, $\Lambda_j$ is well-defined and convex, satisfying
        \[
        \Lambda_j(\mu) = \max_{(\pv,\pm) \in Z_j} \mathcal{F}(\pv,\pm; E(\mu)).
        \]
        Also, $\Lambda_j^*$ is convex and lower semicontinuous, satisfying
        \[
        -\Lambda_j^*(\alpha) = \sup_{(\pv,\pm) \in Z_j(\alpha;\obsfunc)} \mathcal{F}(\pv,\pm; \transmat).
        \]
        In particular, $(\Lambda_j^*)^{-1}[0,\infty) = [\alpha_1,\alpha_2]$ is a nonempty compact interval.
    \end{enumerate}
\end{lemma}
\begin{proof}
    The minimax theorem forms the backbone of the proof.

    (1) The convexity is secured by Lemma \ref{lem:convexity_backbone}.

    (2) The first equality is given in \cite[Proposition 11]{ban2023topological}. To prove the second, note that
    \begin{align*}
        & \sup_{(\pv,\pm) \in Z_j} \left\{ \mathcal{F}_n(\pv,\pm;\transmat): \mathcal{G}_n(\pv,\pm;\obsfunc) = \alpha, \pv[n]=\stdvec{a}\right\} \\
        =& \sup_{\substack{\pm[0:n-1] \in \alphsm^n\\\forall i, \pm[i] \prec \incmat}} \inf_{\mu \in \mathbb{R}} \left[- \mu \alpha - \sum_{i=0}^{n-1} \frac{d-1}{d^{i+1}} \stdvec{a} \left(\prod_{\ell=1}^{n-i-1} \pm[n-\ell]\right) \DKL(\pm[i] || E)\right],
    \end{align*}
    where $\prod_{i=n}^m \pm[i] = \pm[n] \pm[n+1] \cdots \pm[m]$ is set to identity matrix if $m < n$. By fixing $\pm[1:n-1]$ and $\pv[n]$, we may apply the minimax theorem to swap ``$\sup_{\pm[0] \in \alphsm, \pm[0] \prec \incmat}$'' and ``$\inf_{\mu \in \mathbb{R}}$'' and simplify the expression:
    \begin{align*}
        & \sup_{(\pv,\pm) \in Z_j} \left\{ \mathcal{F}_n(\pv,\pm;\transmat): \mathcal{G}_n(\pv,\pm;\obsfunc) = \alpha, \pv[n]=\stdvec{a}\right\} \\
        = & \sup_{\substack{\pm[1:n-1] \in \Upsilon^{n-1}\\\forall i, \pm[i] \prec \incmat}} \inf_{\mu \in \mathbb{R}} \left[- \mu \alpha + \stdvec{a} \left(\prod_{\ell=1}^{n-1} \pm[n-\ell]\right) {\lambda^{(1)}}(\mu)\right. \\
        & \hspace{3em} \left.- \sum_{i=1}^{n-1} \frac{d-1}{d^{i+1}} \stdvec{a} \left(\prod_{\ell=1}^{n-i-1} \pm[n-\ell]\right) \DKL(\pm[i] || E)\right].
    \end{align*}
    We then proceed, by applying the minimax theorem recursively, to swap $\sup_{\pm[i] \in \alphsm, \pm[i] \prec \incmat}$ and $\inf_{\mu \in \mathbb{R}}$ as before to derive the second equality.

    (3) The proof relies on the following estimate: For any $\beta \in \mathbb{R}^{\alphabet}_{>0}$, 
    \[
    \max_{\tv \in \alphpv} \left|\sum_{a \in \alphabet} \tv_a \log \frac{\beta_a}{\tv_a}\right| \le \log\norm{\alphabet} + \max_a \norm{\log \beta_a},
    \]
    which yields an absolute constant $C > 0$ such that $\norm{\DKL(\pm[i] || E(\mu))} \le C \cdot (|\mu|+1)$ and that, in particular, $\mathcal{F}_{n:m}(\pv,\pm;\transmat) = \mathcal{F}_{n:m}(\pv,\pm;E(0)) \le C d^{-n}$. Now given any $a \in \alphabet_0$, express $\lambda^{(n)}(\mu)_a = \mathcal{F}_{n}(\pv,\pm;\transmat)$ for some $(\pv,\pm) \in Z_n$ satisfying $\pv[n] = \stdvec{a}$. We then apply Lemma \ref{lem:probability_surjectivity} to find some $(\pv',\pm') \in Z_n$ such that $({\pv'}^{(0:n)},{\pm'}^{(0:n-1)}) = (\pv[0:n],\pm[0:n-1])$ and $\pv[n+p n_0]=\stdvec{a_0}$. This implies 
    \[
    \lambda^{(n + p n_0)}(\mu)_{a_0} \ge \lambda^{(n)}(\mu)_a + d^{-n} \mathcal{F}_{n:n+p n_0}(\pv',\pm';\transmat) \ge \lambda^{(n - p n_0)}(\mu)_{a_0} - C d^{-n}.
    \]
    For the other inequality, supposing $\lambda^{(n+p n_0)}(\mu)_{a_0} = \mathcal{F}_{n+pn_0}(\pv,\pm;\transmat)$ for some $(\pv,\pm) \in Z_n$ with $\pv[n+pn_0] = \stdvec{a_0}$ yields
    \[
    \max_{a \in \alphabet_0} \lambda^{(n)}_{a} \ge \mathcal{F}_{n}(\pv,\pm;\transmat) + \mathcal{F}_{n:n+pn_0}(\pv,\pm;\transmat) - C d^{-n} = \lambda^{(n + p n_0)}_{a_0} - C d^{-n}.
    \]
    Combining the above proves the proposed bound.

    (4) The existence of the limit in $\Lambda_j$ as well as the proposed equality is justified by (3), while its convexity follows from (1).

    For $\Lambda_j^*$, we note that its convexity and lower semicontinuity are consequences of the Legendre transform of $\Lambda_j$ (see, e.g., \cite[Lemma 2.2.5 (a)]{Dembo2009-qs}). We proceed to prove the first identity. Note that by definition,
    \begin{equation} \label{eq:strong_duality_LDP}
        \sup_{(\pv,\pm) \in Z_j} \left\{ \mathcal{F}(\pv,\pm;\transmat): \mathcal{G}(\pv,\pm;\obsfunc) = \alpha\right\} = \sup_{(\pv,\pm) \in Z_j} \inf_{\mu \in \mathbb{R}} \left[-\mu \alpha + \Lambda_j(\mu)\right] \le - \Lambda_j^*(\mu),
    \end{equation}
    which is also known as the \emph{weak duality} in the context of optimization theory. Hence, it remains to prove the other inequality. For simplicity, we write 
    \[
    f_j^{(n)}(\alpha) = \sup_{(\pv,\pm) \in Z_j(\alpha;A), \pv[n]=\stdvec{a_0}} \mathcal{F}_n(\pv,\pm;\transmat) \quad \text{and} \quad f_j(\alpha) = \sup_{(\pv,\pm) \in Z_j(\alpha;A)} \mathcal{F}(\pv,\pm;\transmat),
    \]
    and take $n_0$ to be an integer as in Lemma \ref{lem:probability_surjectivity}.
    
    First, we claim that there is a compact interval $\emptyset \ne [\alpha_1',\alpha_2'] \subseteq [\alpha_1,\alpha_2]$ such that $\sup_{(\pv,\pm) \in Z_j(\alpha;\obsfunc)} \mathcal{F}(\pv,\pm; \transmat) > -\infty$ if and only if $\alpha \in [\alpha_1',\alpha_2']$. The nonemptiness and compactness of the domain of $f_j$ follows from the compactness of $Z_j$ and the continuity of $\mathcal{F}$ and $\mathcal{G}$. Based on the observation, if 
    \[
    \alpha_1' := \mathcal{G}(\pv',\pm';A) < \mathcal{G}(\pv'',\pm'';A) =: \alpha_2'
    \]
    are the extreme points of the domain of $f_j$ with $(\pv',\pm'), (\pv'',\pm'') \in Z_j$, then for all $n \in \mathbb{N}$ there exist, by Lemma \ref{lem:probability_surjectivity}, approximations $(\tv',\tm'), (\tv'',\tm'') \in Z_j$ such that
    \[
    \begin{cases}
        ((\tv')^{(p(n+n_0)+j+1:\infty)},(\tm')^{(p(n+n_0)+j:\infty)}) = ((\tv'')^{(p(n+n_0)+j+1:\infty)},(\tm'')^{(p(n+n_0)+j:\infty)}), \\
        ((\tv')^{(0:pn+j)},(\tm')^{(0:pn+j-1)}) = ((\pv')^{(0:pn+j)},(\pm')^{(0:pn+j-1)}), \\
        ((\tv'')^{(0:pn+j)},(\tm'')^{(0:pn+j-1)}) = ((\pv'')^{(0:pn+j)},(\pm'')^{(0:pn+j-1)}),
    \end{cases}
    \]
    which in turn implies that for all sufficiently large $n$,
    \[
    \mathcal{F}(\tv',\tm'; \transmat) < \alpha_1' + C d^{-pn-j} < \alpha_2' - C d^{-pn-j} < \mathcal{F}(\tv'',\tm''; \transmat).
    \]
    Since it is not hard to construct a continuous curve $\gamma: [0,1] \to Z_j$ connecting $(\tv',\tm')$ and $(\tv'',\tm'') \in Z_j$, it follows from the intermediate value theorem that
    \[
    Z_j(\alpha;\obsfunc) \ne \emptyset \text{ for all } \alpha \in [\alpha_1' + C d^{-pn-j}, \alpha_2' - C d^{-pn-j}],
    \]
    proving the claim when $n \to \infty$. 
    
    Next, we show that \eqref{eq:strong_duality_LDP} holds for all $\alpha \notin [\alpha_1,\alpha_2]$, or equivalently, $[\alpha_1',\alpha_2'] = [\alpha_1,\alpha_2]$. Suppose $\alpha > \alpha_2'$ (similar for $\alpha < \alpha_1'$), so the uniform convergence of $\mathcal{G}_n(\pv,\pm;A)$ implies $f_j^{(pn+j)}(\alpha) = - \infty$ for all sufficiently large $n$. Therefore, by (2) and (3),
    \begin{align}\label{eq:Legendre_infty}
    \begin{aligned}
        -\infty &= f_j^{(p (n+n_0) + j)}(\alpha) = \inf_{\mu \in \mathbb{R}} \left[- \mu \alpha + \lambda^{(p (n+n_0) +j)}(\mu)_{a_0}\right] \\
        &\ge \inf_{\mu \in \mathbb{R}} \left[- \mu \alpha_0 + \max_{a \in \alphabet_0} \lambda^{(pn+j)}(\mu)_{a} - C d^{-(pn+j)} (|\mu|+1)\right] \\
        &\ge \inf_{\alpha': \norm{\alpha' - \alpha} \le 2 C d^{-(pn+j)}} -\Lambda_j^*(\alpha') - 2 C d^{-(pn+j)}.
    \end{aligned}
    \end{align}
    Since the estimate holds for all $\alpha > \alpha_1'$ and all large $n$, $\Lambda_j^*$ is convex, and $\Lambda_j^*(\alpha_2') < \infty$ is finite, we deduce that $\Lambda_j^*(\alpha) = \infty$ for all $\alpha > \alpha_2'$, that is, $\alpha_2 = \alpha_2'$. 
    
    Finally, it remains to show that \eqref{eq:strong_duality_LDP} holds for $\alpha \in [\alpha_1,\alpha_2]$. The case $\alpha \in (\alpha_1,\alpha_2)$ is proved similarly as above. Let $\delta > 0$. By uniform convergence of $\mathcal{G}_n$, $f_j^{(pn+j)}(\alpha) < \infty$ for all sufficiently large $n$. It follows from (3) that for all large $n$,
    \begin{align}\label{eq:Legendre_finite}
    \begin{aligned}
        \sup_{\alpha': |\alpha'-\alpha| < \delta} f_j(\alpha') &\ge f_j^{(p (n+n_0) + j)}(\alpha) - C d^{-(p(n+n_0)+j)} \\
        &= \inf_{\mu \in \mathbb{R}} \left[- \mu \alpha + \lambda^{(p(n+n_0)+j)}(\mu)_{a_0} - C d^{-(p(n+n_0)+j)}\right] \\
        &\ge \inf_{\mu \in \mathbb{R}} \left[- \mu \alpha + \max_{a \in \alphabet_0} \left[\lambda^{(pn+j)}(\mu)_{a} - 2 C d^{-pn+j}(|\mu|+1)\right]\right] \\
        &\ge \inf_{\mu \in \mathbb{R}} \left[- \mu \alpha + \Lambda_j(\mu) - 3 C d^{-(pn+j)} \norm{\mu}\right] - 3 C d^{-(pn+j)} \\
        &\ge \inf_{\alpha': \norm{\alpha' - \alpha} \le 3 C d^{-(pn+j)}} -\Lambda_j^*(\alpha') - 3 C d^{-(pn+j)}.
    \end{aligned}
    \end{align}
    Since $f_j$ and $-\Lambda_j^*$ are continuous on $(\alpha_1,\alpha_2)$, the equality \eqref{eq:strong_duality_LDP} is proved by first letting $n \to \infty$ and then $\delta \to 0$. If $\alpha \in \{\alpha_1,\alpha_2\}$, the discussion is divided into the following two cases. If $\alpha_1 = \alpha_2$, then $f_j(\alpha_1) = 0$ as $\max_{(\pv,\pm) \in Z_j} \mathcal{F}(\pv,\pm;\transmat)$ has maximum $0$ by nonnegativity of Kullback-Leibler divergence, and $0 = f_j(\alpha_1) \ge -\Lambda_j^*(\alpha_1) \ge \Lambda_j^*(0) = 0$. If $\alpha_1 < \alpha_2$, it suffices to show that $f_j$ is continuous at $\alpha_1$ and $\alpha_2$, since under the circumstances both $f_j$ and $-\Lambda_j^*$ are continuous at $[\alpha_1,\alpha_2]$ and therefore coincide. This follows similarly from the intermediate value theorem as before. 
\end{proof}

\begin{proof}[Proof of Theorem \ref{thm:1}]
    For conciseness, we prove only the case $j=0$. Observing that the ratio $\frac{d-1}{d^{n-i+1}} / \frac{\norm{\level{i}}}{\norm{\lattice{n}}} = 1-d^{-n-1}$ is uniform for $0 \le i \le n$, we write
    \begin{gather*} 
        \widetilde{\mathcal{F}}_{n}(\mathsf{p},\mathsf{P};\transmat) := -\sum_{i=0}^{n-1} \frac{\norm{\level{n-i}}}{\norm{\lattice{n}}} \mathsf{p}^{(i+1)} \DKL(\pm[i] || \transmat) = \frac{\mathcal{F}_{n}(\mathsf{p},\mathsf{P};\transmat)}{1-d^{-n-1}}, \\
        \widetilde{\mathcal{G}}_{n}(\mathsf{p},\mathsf{P};\obsfunc)  := \sum_{i=0}^{n-1} \frac{\norm{\level{n-i}}}{\norm{\lattice{n}}} \pv[i+1] (\pm[i] \odot \log \obsfunc) \onevec = \frac{\mathcal{G}_{n}(\mathsf{p},\mathsf{P};\obsfunc)}{1-d^{-n-1}}.
    \end{gather*}
    Notably, the sample mean $Y_{pn} = \widetilde{\mathcal{G}}_{n}(\pv,\pm;\obsfunc)$ if the Markov chain $X = (X_g)_{g \in \tree}$ satisfies $(\dv{0}[n](X), \trm{0}[n-1](X)) = (\cev{\pv[0:n]},\cev{\pm[0:n-1]})$, which follows from a combinatorial argument: In level $i$ ($0 \le i < pn$), there exist $\norm{\level{i}} \dv{i}(X)_a = \norm{\level{i}} \pv[n-i]_a$ vertices associated with state $a$, which are responsible for $(d \cdot \norm{\level{i}} \dv{i}(X)_a) \cdot \trm{i}(X)_{a,b} = \norm{\level{i+1}} \pv[n-i]_a \pm[n-i-1]_{a,b}$ transitions from state $a$ to $b$ between level $i$ and $i+1$.

    Similar to the heuristic proof of Theorem \ref{thm:Cramer}, we first infer from Lemmas \ref{lem:dist_set_growth_rate} and \ref{lem:uniform_equiconvergence} that
    \begin{equation} \label{eq:combinatorial_manipulation}
        \begin{aligned}
            & \frac{1}{\norm{\lattice{pn}}} \log \mathbb{P}\left(\left.Y_{pn} \in S \right|X_{\epsilon} = a_0\right) \\
            = & \sup_{\alpha \in S} \left\{\frac{1}{\norm{\lattice{pn}}} \log \mathbb{P}\left(\widetilde{\mathcal{G}}_{pn}(\pv,\pm;\obsfunc) = \alpha\right): \pv[pn]=\stdvec{a_0},\right. \\
            & \hspace{5em} \left.\vphantom{\sum_{i=0}^{pn-1}} (\cev{\pv[0:pn]},\cev{\pm[0:pn-1]}) = (\dv{0}[pn](X), \trm{0}[pn-1](X))\right\} + O\left(\frac{\log \norm{\lattice{pn}}}{\norm{\lattice{pn}}}\right) \\
            = & \sup_{\alpha \in S} \left\{\widetilde{\mathcal{F}}_{n}(\pv,\pm; \transmat): \pv[pn]=\stdvec{a_0}, (\cev{\pv[0:pn]},\cev{\pm[0:pn-1]}) = (\dv{0}[pn](X), \trm{0}[pn-1](X)), \right. \\
            & \hspace{5em} \left.\vphantom{\sum_{i=0}^{pn-1}} \widetilde{\mathcal{G}}_{n}(\pv,\pm;\obsfunc) = \alpha\right\} + O\left(\frac{\log \norm{\lattice{pn}}}{\norm{\lattice{pn}}}\right),
        \end{aligned}
    \end{equation}
    To prove the upper bound of the theorem, we consider the closed $(1/k)$-neighborhood $S_k = \{\alpha' \in \mathbb{R}: \inf_{\alpha \in S} |\alpha' - \alpha| \le 1/k\}$ of $S$. Taking advantage of the uniform convergence and the uniform boundedness of $\mathcal{F}_n$ and $\mathcal{G}_n$, we can bound \eqref{eq:combinatorial_manipulation}, for sufficiently large $n$, from above by
    \begin{align*}
        & \sup_{\alpha \in S_k} \left\{\mathcal{F}(\pv,\pm; \transmat): (\pv,\pm) \in Z_0(\alpha;\obsfunc)\right\} + o\left(1\right),
    \end{align*}
    which, by Lemma \ref{lem:asymptotic_property} (4), coincides further with
    \begin{align*}
        \sup_{\alpha \in S_k} \inf_{\mu \in \mathbb{R}} \left[-\mu \alpha + \Lambda_0(\mu)\right] + o\left(1\right).
    \end{align*}
    The desired inequality is obtained by first letting $n \to \infty$ and then $k \to \infty$. For the lower bound, we analogously take $S_k = \{\alpha' \in \mathbb{R}: \inf_{\alpha \notin S} |\alpha' - \alpha| \ge 1/k\}$. By Lemma \ref{lem:dense_pattern} and the equicontinuity of $\mathcal{F}_n$ and $\mathcal{G}_n$, one derives the following uniform lower bound for \eqref{eq:combinatorial_manipulation}:
    \begin{align*}
        \sup_{\alpha \in S_k} \left\{\frac{\mathcal{F}(\pv,\pm; \transmat)}{1-d^{-pn-1}}: (\pv,\pm) \in Z^{(n)}((1-d^{-pn-1})\alpha;\obsfunc), \pv[n] = \stdvec{a_0}\right\} + o\left(1\right),
    \end{align*}
    Letting $n \to \infty$, the lower limit of the above (possibly $-\infty$) is further bounded from below by
    \begin{align*}
        \sup_{\alpha \in S_{k+1}} \left\{\mathcal{F}(\pv,\pm; \transmat): (\pv,\pm) \in Z_0(\alpha;\obsfunc)\right\},
    \end{align*}
    or, equivalently by Lemma \ref{lem:asymptotic_property} (4),
    \begin{align*}
        \sup_{\alpha \in S_{k+1}} \inf_{\mu \in \mathbb{R}} \left[-\mu \alpha + \Lambda_0(\mu)\right].
    \end{align*}
    The proof is concluded by letting $k \to \infty$.
\end{proof}
As a special case, the following corollary simplifies the expression of Theorem \ref{thm:1} when focusing on intervals.
\begin{corollary} \label{cor:1}
    Suppose that $X$, $Y$, and $\Lambda_j^*(\alpha)$ are as in Theorem \ref{thm:1}. Then,
    \begin{equation*}
        \begin{aligned}
            & \lim_{\varepsilon \to 0^{+}} \lim_{n \to \infty} \frac{1}{\norm{\lattice{pn+j}}} \log \mathbb{P}\left(\left.Y_{pn+j} \in (\alpha-\varepsilon,\alpha+\varepsilon)\right|X_{\epsilon}=a_0\right) = -\Lambda^{*}_{j}(\alpha).
        \end{aligned}
    \end{equation*}
\end{corollary}
\begin{proof}
    It is a consequence of Theorem \ref{thm:1} with the existence of the limit justified by continuity $\Lambda^{*}_j(\alpha)$, a consequence of convexity and lower semicontinuity.
\end{proof}
\begin{proof}[Proof of Theorem \ref{thm:2}]
    For simplicity, we only give the proof of the case $j=0$.

    To begin, we note that the existence and uniqueness of $(\initvec^{(i)})_{i \in \mathbb{Z}}$ is a consequence of the Perron-Frobenius theorem (see, e.g., \cite[Theorems 8.2.8 and 8.4.4]{horn2012matrix}). Precisely, according to the characterization \eqref{eq:label_decomposition}, we can express $\transmat$ as a superdiagonal block matrix $\transmat = (\transmat_{i, j})_{0 \le i,j \le p-1}$ with each $\transmat_{i, j} \in \mathbb{R}_{\ge 0}^{\alphabet_i \times \alphabet_j}$, and $\transmat^p$ is a diagonal block matrix whose diagonal elements are primitive matrices. The existence and uniqueness follows immediately from the cited theorem.
    
    We first show that any maximizer $(\pv,\pm)$ of $\max_{(\pv,\pm) \in Z_0} \mathcal{F}(\pv,\pm;\transmat)$ satisfies
    \begin{equation} \label{eq:unique_maximizer}
        \mathcal{G}(\pv,\pm;\obsfunc) = \mathcal{G}(\pv^*,\pm^*;\obsfunc)
    \end{equation}
    where
    \[
    (\pv^*,\pm^*) = ((\initvec^{(0)},\initvec^{(1)},\cdots,\initvec^{(p-1)})^{\mathbb{Z}_+},(\transmat)^{\mathbb{Z}_+}) \in Z_0.
    \]
    Notably, $(\pv^*,\pm^*)$ is indeed a maximizer: For every $(\tv,\tm) \in Z_0$,
    \begin{equation} \label{eq:uniqueness}
        \DKL(\tm[i] || \transmat) \ge 0 \quad \text{ and } \quad \DKL(\tm[i] || \transmat) = 0 \text{ iff } \tm[i]=\transmat,
    \end{equation}
    from which the maximality follows. To demonstrate \eqref{eq:unique_maximizer}, we apply \eqref{eq:uniqueness} to derive that for any optimal solution $(\mathsf{p},\mathsf{P})$,
    \[
        \mathsf{p}^{(pN+i)} = \mathsf{p}^{(p(N+1)+i)} \transmat^p \text{ for all } i \ge 0, n \ge 0.
    \]
    Recall that $\transmat^p$ is a diagonal block matrix whose diagonal elements are primitive. The Perron-Frobenius theorem guarantees that
    \[
        \mathsf{p}^{(pN+i)}=\lim_{n \to \infty} \mathsf{p}^{(pN+i+pn)} \transmat^{pn} =\initvec^{(i)}.
    \]
    Since $\mathsf{p}^{(pn+i)}_{a}=\initvec^{(i)}_{a}$ is positive if and only if $a \in \alphabet_{-i}$ and $n,i \ge 0$, once again we can apply \eqref{eq:uniqueness} to derive $\pm[pn+i]_{a,b}=\transmat_{a,b}$ if $(a,b) \in \alphabet_{-i-1} \times \alphabet_{-i}$, and the claim is thus proved. Equivalently, by Lemma \ref{lem:asymptotic_property}, our claim implies $-\Lambda^{*}_0(\alpha)$ admits a unique maximum point 
    \[
    \alpha = \alpha^* := \mathcal{G}(\pv^*,\pm^*;\obsfunc) = \sum_{i=0}^{p-1} \frac{d^{-i}}{\sum_{\ell=0}^{p-1} d^{-\ell}} \sum_{a, b \in \alphabet} \initvec^{(i+1)}_a \transmat_{a,b} \log \obsfunc_{a,b},
    \]
    at which $-\Lambda^{*}_0$ attains its maximum $0$.

    The rest of the proof is an application of the Borel-Cantelli lemma. Observe that $\delta(\varepsilon) = \sup_{\norm{\alpha-\alpha^*} \ge \varepsilon} - \Lambda^{*}_{0}(\alpha) < 0$ for all $\varepsilon > 0$ due to the uniqueness of $\alpha^*$. This secures the existence of $N \in \mathbb{N}$ such that for all $n \ge N$,
    \begin{align*}
        \mathbb{P}\left(\left.\left\vert Y_{pn} - \alpha^*\right\vert > \varepsilon\right|X_{\epsilon} = a_0\right) \le e^{\delta(\epsilon) \norm{\lattice{pn}}}
    \end{align*}
    Summing over $n$, we obtain
    \[
    \sum_{n=N}^{\infty} \mathbb{P}\left(\left.\left\vert Y_{pn} - \alpha^*\right\vert > \varepsilon\right|X_{\epsilon} = a_0\right) \le \sum_{n=N}^{\infty} e^{\delta(\varepsilon) \norm{\lattice{pn}}} < \infty,
    \]
    which, by the Borel-Cantelli lemma, implies
    \[
    \mathbb{P}\left(\left.\limsup_{n \to \infty} \{\left\vert Y_{pn} - \alpha^*\right\vert > \varepsilon\}\right|X_{\epsilon} = a_0\right) = 0.
    \]
    This completes the proof since $\varepsilon > 0$ is arbitrary.
\end{proof}
\begin{example}
    Let $d=2$ and $\obsfunc$, $\transmat$ be matrices defined as
    \[
    \transmat = \begin{bmatrix}
        \frac{1}{2} & \frac{1}{2} \\
        1 & 0
    \end{bmatrix} \quad \text{and} \quad \obsfunc = \begin{bmatrix}
        1 & 1 \\
        2 & 0
    \end{bmatrix}.
    \]
    Since $\transmat$ is primitive, Theorem \ref{thm:1} ensures that the averages
    \[
    Y_n \to \sum_{a,b \in \alphabet} \initvec_{a} \transmat_{a,b} \log \obsfunc_{a,b} \quad \mathbb{P}\text{-a.s.}
    \]
    and satisfy the large deviation principle with rate $\Lambda_0^*(\alpha)$. Numerically, $-\Lambda^{*}_0(\alpha)$ is computed and depicted in Figure \ref{fig:example_1}. We note that $-\Lambda^{*}_0(\alpha)$ is finite if and only if $\alpha \in [0, \frac{2}{3} \log 2]$, where $\frac{2}{3} \log 2 \approx 0.4621$. Furthermore, the maximum point is $\alpha = \frac{1}{3} \log 2 \approx 0.2310$.
    \begin{figure}
        \centering
        \includegraphics[width=0.5\textheight]{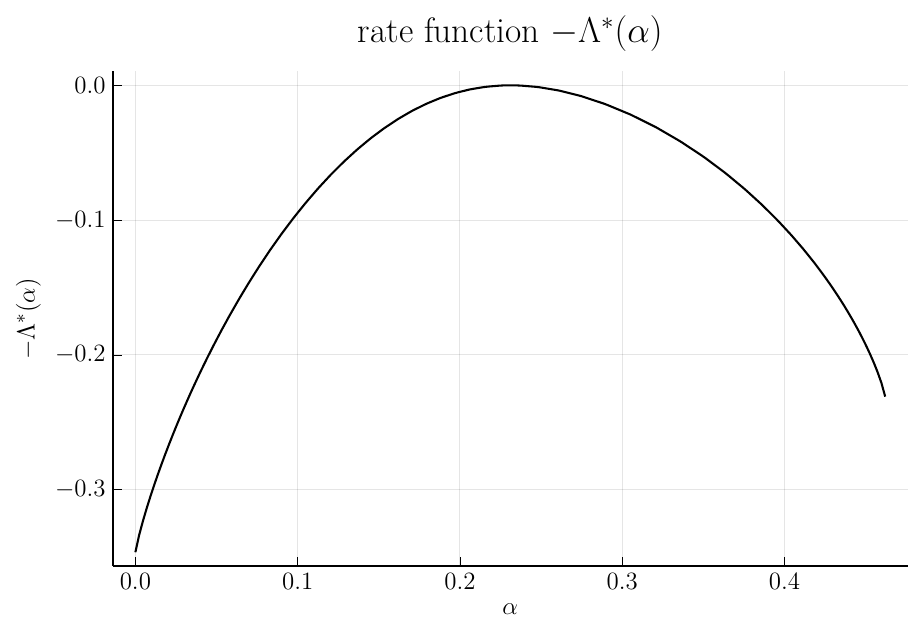}
        \caption{rate function $\Lambda^*_0$}
        \label{fig:example_1}
    \end{figure}
\end{example}
In the following, we present yet another example of irreducible $\transmat$ in order to (a) demonstrate the upper and lower limit of the conditional sample mean might differ, and (b) compare the conditional sample mean, unconditional sample mean, and their expectations.
\begin{example} \label{ex:examples}
     We consider $d=2$ and
    \[
    \obsfunc = \transmat = \begin{bmatrix}
        0 & \frac{1}{3} & \frac{2}{3} \\
        1 & 0 & 0 \\
        1 & 0 & 0
    \end{bmatrix}, \initvec = \begin{bmatrix}\frac{1}{6} & \frac{1}{3} & \frac{1}{2}\end{bmatrix}.
    \]
    for which the conditional sample mean associated with $\alphabet_0:=\{0\}$ and on $\alphabet_1:=\{1,2\}$ are plotted in Figure \ref{fig:irr}, respectively. To address (a), let $\alpha_0^*$ and $\alpha_1^*$ be the maximum points of $\Lambda^*_0$ and $\Lambda^*_1$, respectively. Theorem \ref{thm:1} then guarantees that almost surely
    \[
    \alpha^-:=\liminf_{n \to \infty} Y_n = \alpha_1^*
    \quad \text{and} \quad \alpha^+:=\limsup_{n \to \infty} Y_n = \alpha_0^*.
    \]
    Regarding (b), the rate function $\Lambda^*$ of the unconditional sample mean is 
    \[
    \Lambda^*(\alpha) = \max \{\Lambda^*_j(\alpha): \initvec|_{\alphabet_j} \ne 0\} = \max \{\Lambda^*_0(\alpha),\Lambda^*_1(\alpha)\},
    \]
    which attains its maximum at $\alpha_0^*$ and $\alpha_1^*$. For the expectation of sample mean, we deduce by the dominated convergence theorem that
    \begin{gather*}
        \beta^-:=\liminf_{n \to \infty} \mathbb{E}\left(Y_n\right) = \min_{0 \le i \le 1} \sum_{j=0}^{1} \sum_{a \in \alphabet_j} \initvec_a \alpha_{i+j}^*, \\
        \beta^+:=\limsup_{n \to \infty} \mathbb{E}\left(Y_n\right) = \max_{0 \le i \le 1} \sum_{j=0}^{1} \sum_{a \in \alphabet_j} \initvec_a \alpha_{i+j}^*.
    \end{gather*}
    Notably, $\alpha^- < \beta^- < \beta^+ < \alpha^+$, as opposed to the case when $\transmat$ is primitive and the four numbers coincide. This indicates that the conditional sample mean provides more information than the unconditional one or the expectation. In fact, in some extreme cases, say
    \[
    \obsfunc = \transmat = \begin{bmatrix}
        0 & 1 & 1 \\
        \frac{1}{2} & 0 & 0 \\
        \frac{1}{2} & 0 & 0
    \end{bmatrix}, \initvec = \begin{bmatrix}
        \frac{1}{2} & \frac{1}{4} & \frac{1}{4}
    \end{bmatrix},
    \]
    we observe that
    \[
    \alpha^-=\frac{1}{3} \log 2, \qquad \alpha^+=\frac{2}{3} \log 2, \qquad \beta^-=\beta^+=\frac{1}{2} \log 2,
    \]
    and that for every $t \in \tshift[][\incmat] (= \mathrm{supp} \mathbb{P})$, the sample mean can be explicitly calculated as
    \[
    Y_n(t) \in \left\{\alpha^- + o(1), \alpha^+ + o(1)\right\}.
    \]
    Nevertheless, it is interesting yet unsurprising that not a single labeled tree $t \in \tshift[][\incmat]$ admits a sample mean approaching $\beta^+ = \beta^-$.
    \begin{figure}
        \centering
        \includegraphics[width=0.5\textheight]{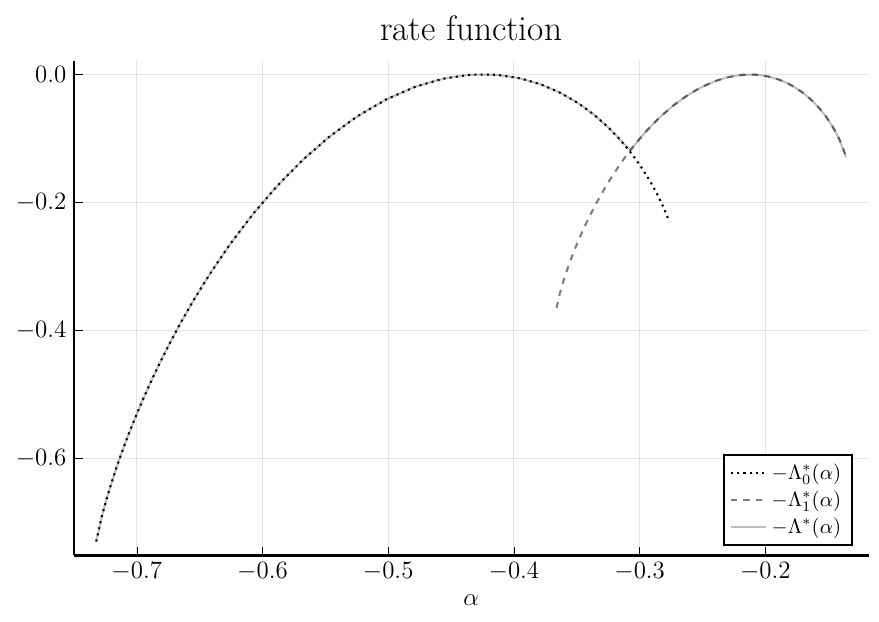}
        \caption{rate function $\Lambda^*_0$, $\Lambda^*_1$, and $\Lambda^*$}
        \label{fig:irr}
    \end{figure}
\end{example}

\section{Hausdorff dimensions of irreducible Markov hom tree-shifts} \label{sec:Hausdorff_dimension_irr}
The section is dedicated to presenting the Hausdorff dimension formula for Markov hom tree-shifts $\tshift[][\incmat] \subset \alphabet^{\tree}$ when $\incmat$ is an irreducible matrix. We begin by considering the following metric on $\alphabet^{\tree}$:
\begin{equation} \label{eq:tree-shift_metric}
    D(x,y) = e^{-\sup\{ \norm{\lattice{n}}: x|_{\lattice{n}} = y|_{\lattice{n}} \}}.
\end{equation}
Notably, this metric coincides with the canonical metric when $\tree = \mathbb{Z}_+$ is degenerate and the choice is due to the its intimate relation with topological entropy of tree-shifts, which was considered by Petersen and Salama \cite{Petersen2018b,Petersen2020} as a natural generalization of its counterpart for one-sided subshifts. See also \cite{Petersen2021} for related results for asymptotic pressure of tree-shifts, a generalized form of the topological entropy. Specifically, the \emph{topological entropy} of a Markov hom tree-shift $\tshift[][\incmat]$ is defined by
\[
h_{top}(\tshift[][\incmat]) := \limsup_{n \to \infty} \frac{\log \norm{\block{n}(\tshift[][\incmat])}}{\norm{\lattice{n}}},
\]
where the sequence actually converges as established in the aforementioned work. Using this definition and the existence of the limit, we observe the following equalities regarding box-counting dimensions:
\begin{align*}
    &\overline{\dim}_B \tshift[][\incmat] := \limsup_{r \to \infty} -\frac{\log \mathcal{N}_r(\tshift[][\incmat])}{\log r} = d \cdot h_{top}(\tshift[][\incmat]), \\
    &\underline{\dim}_B \tshift[][\incmat] := \liminf_{r \to \infty} -\frac{\log \mathcal{N}_r(\tshift[][\incmat])}{\log r} = h_{top}(\tshift[][\incmat]),
\end{align*}
where $\mathcal{N}_r(\tshift[][\incmat])$ denotes the minimal number of closed $r$-balls needed to cover the set $\tshift[][\incmat]$. It is worth mentioning that the multiplicative difference by $d$ arises from the fact that closed $r$-balls $B_r(t)$ are aliases of the same set for all $r \in [e^{-\norm{\lattice{n}}}, e^{-\norm{\lattice{n-1}}}) = [e^{-\frac{d^{n+1}-1}{d-1}}, e^{-\frac{d^{n}-1}{d-1}})$, causing $\mathcal{N}_r(\tshift[][\incmat])$ to remain constant over the interval. As a natural follow-up question, it would be interesting to determine the packing and Hausdorff dimensions of the Markov hom tree-shifts, where the former is a known to coincide with $\overline{\dim}_B \tshift[][\incmat]$ when $\incmat$ is irreducible (see Appendix \ref{sec:packing_dim}) and the latter is our primary focus of this section. As mentioned in the introduction, this goal is achieved by finding a nonlinear transfer operator to link the Hausdorff dimension of $\tshift[][\incmat]$ to the eigenvalue of the operator.


\subsection{Nonlinear Perron-Frobenius theory} \label{sec:Perron-Frobenius}
The nonlinear Perron-Frobenius theory, as its name suggests, studies the eigenspace of a class of (not necessarily linear) mappings and reproduces several classical results regarding nonnegative primitive/irreducible matrix transformations. In this work, we will consider such an analysis under the framework of \cite{Lemmens2012}, of which the setting can be described as follows. Let $f: \mathbb{R}_{\ge 0}^{\alphabet} \to \mathbb{R}_{\ge 0}^{\alphabet}$ be a continuous function. It is said to be \emph{order-preserving} if $f(x) \le f(y)$ for every $x \le y$, and it is called \emph{homogeneous} if $f(\alpha x) = \alpha f(x)$ for all $\alpha \in \mathbb{R}_{\ge 0}$, where the comparison of two vectors via $\le$ is taken entrywise. In addition, we say that $f$ is \emph{multiplicatively convex} if $\log \circ f \circ \exp$ is a convex function on $\mathbb{R}_{> 0}^{\alphabet}$ (see Definition \ref{def:convexity}), where $\log$ and $\exp$ are applied entrywise as explained in Section \ref{sec:preliminaries}. Within this framework, we establish that the operator $\mathcal{L}_{\incmat,r}$ (recall \eqref{eq:transfer_operator}) possesses all the aforementioned characteristics and that its eigenspace can be characterized as follows. We note that throughout the section, $\Norm{\cdot}$ denotes the $1$-norm.
\begin{proposition} \label{prop:eigenspace}
    Suppose $\incmat$, $\mathcal{L}_{\incmat,r}$, and $r \in \mathcal{R}_{p,d}$ be as in Theorem \ref{thm:Hausdorff_dim}. Then, the following properties hold.
    \begin{enumerate}
        \item $\mathcal{L}_{\incmat,r}: \mathbb{R}_{\ge 0}^{\alphabet} \to \mathbb{R}_{\ge 0}^{\alphabet}$ is continuous, order-preserving, homogeneous, analytic on $\mathbb{R}_{>0}^{\alphabet}$, and multiplicatively convex.
        \item Suppose $\mathcal{L}_{\incmat,r}$ maps $\mathcal{C}' = \{x \in \mathbb{R}_{\ge 0}^{\alphabet}: x_a = 0 \text{ if } a \notin \alphabet'\}$ into itself for some $\alphabet' \subseteq \alphabet$. Then, there exists an eigenvector $v \in \mathcal{C}'$ associated with eigenvalue $\rho(\mathcal{L}_{\incmat,r}|_{\mathcal{C}'})$ such that 
        \[
        \limsup_{n \to \infty} \Norm{\mathcal{L}^n_{\incmat,r}(w)}^{1/n} \le \rho(\mathcal{L}_{\incmat,r}|_{\mathcal{C}'}) \text{ for all } w \in \mathcal{C}'.
        \]
        In particular, if $\incmat$ is irreducible, for each $\alphabet_j \in \mathcal{P}(\incmat)$ there exists a unique (up to scaling) eigenvector $v^{(j)} \in \mathcal{C}_j$, whose associated eigenvalue $\rho_{\alphabet_j}(\mathcal{L}_{\incmat,r})$ satisfying the stated property on $\mathcal{C}_j$, and $v^{(j)}_a > 0$ for all $\alphabet_j$.
        \item If $\incmat$ is irreducible, then for every $x \in \mathcal{C}_j^+ := \{x \in \mathcal{C}_j: x_a > 0 \text{ if } a \in \alphabet_j\}$ there exists $0 < \theta < 1$ such that
        \[
        \limsup_{n \to \infty} \Norm{\frac{\mathcal{L}_{\incmat,r}^n(x)}{\Vert\mathcal{L}_{\incmat,r}^n(x)\Vert} - \frac{v^{(j)}}{\Vert v^{(j)}\Vert}}^{1/n} = \theta,
        \]
        where $v^{(j)} \in \mathcal{C}_j^+$ is the unique eigenvector as in (2).
    \end{enumerate}
\end{proposition}
\begin{proof}
    (1) The first four conditions follow from a routine check of definition, while the multiplicative convexity follows from Lemma \ref{lem:convexity_backbone} applied to $\Psi_{\incmat,s}$ ($s > 0$). which extends naturally to $\mathcal{L}_{\incmat,r} = \Psi_{\incmat,r_{p-1}} \circ \cdots \circ \Psi_{\incmat,r_{0}}$.

    (2) According to \cite[Corollary 5.4.2]{Lemmens2012}, $\mathcal{L}_{\incmat,r}$ admits an eigenvector $v \in \mathcal{C}'$ such that $\mathcal{L}_{\incmat,r}(v) = \rho(\mathcal{L}_{\incmat,r}|_{\mathcal{C}'}) \cdot v$, while \cite[Theorem 5.3.1]{Lemmens2012} asserts that $\limsup_{n \to \infty} \Norm{\mathcal{L}_{\incmat,r}(w)}^{1/n} \le \rho(\mathcal{L}_{\incmat,r}|_{\mathcal{C}'})$ for all $w \in \mathcal{C}'$.  We observe that for any $x \in \mathbb{R}_{\ge 0}^{\alphabet}$, $\Psi_{\incmat,s}(x)_a > 0$ if and only if $(\incmat x)_a > 0$. Recursively applying this property one deduces that $\mathcal{L}_{\incmat,r}(x)_a > 0$ if and only if $(\incmat^p x)_a > 0$. Now if $\incmat$ is irreducible, $\mathcal{P}(\incmat)$ is a partition and $\incmat_{a,b}>0$ if and only $(a,b) \in \alphabet_j \times \alphabet_{j+1}$ for some $j \in [p]$. Notably, $\mathcal{L}_{\incmat,r}$ maps $\mathcal{C}_j$ to $\mathcal{C}_j$, each of which hence admits an eigenvector $v^{(j)} \in \mathcal{C}_j$. Moreover, $v^{(j)}_a > 0$ for all $a \in \alphabet_j$ as every state $a$ plays the role of $a_0$ in \eqref{eq:tree_assumption}. By denoting the natural inclusion as $\iota_j: \mathcal{C}_j \to \mathbb{R}_{\ge 0}^{\alphabet}$, we have that $\iota_j^{-1} \circ \mathcal{L}_{\incmat,r} \circ \iota_j$ maps $\mathbb{R}_{\ge 0}^{\alphabet_j}$ into $\mathbb{R}_{\ge 0}^{\alphabet_j}$ and the derivative of $\iota_j^{-1} \circ \mathcal{L}_{\incmat,r} \circ \iota_j$ at $\iota_j^{-1}(v^{(j)})$ is primitive. Therefore, by \cite[Corollary 6.5.8]{Lemmens2012}, $\iota_j^{-1}(v^{(j)})$ is unique up to scaling, and thus so is $v^{(j)}$.  

    (3) We note that all the properties of $\mathcal{L}_{\incmat,r}$ in (1) are inherited by $\iota_j^{-1} \circ \mathcal{L}_{\incmat,r} \circ \iota_j$. Hence, the existence of $\theta$ follows as a consequence of \cite[Theorem 4.7]{lins2023convergence}. 
\end{proof}
\subsection{Some lemmas} \label{sec:some_lemmas}
For conciseness, we suppress, unless mentioned otherwise, $\incmat$ from $\mathcal{F}_{n:m}(\pv,\pm;\incmat)$ for the rest of the paper. We first state a number of technical lemmas, whose proofs are postponed until the appendix. These lemmas establish the desired minimax properties, bridging the upper bound of the Hausdorff dimension obtained through efficient covering with the lower bound obtained via the mass distribution principle, as discussed in the coming subsections.
\begin{lemma} \label{lem:periodic_duality}
    The following minimax property holds.
    \begin{equation} \label{eq:periodic_duality}
        \max_{\substack{(\pv,\pm) \in Z_0: \\ p\text{-periodic}}} \min_{s \in \Gamma_{[p]}} \sum_{m=0}^{p -1} s_m \mathcal{F}_{m:\infty}(\pv,\pm) = \min_{s \in \Gamma_{[p]}} \max_{\substack{(\pv,\pm) \in Z_0: \\ p\text{-periodic}}} \sum_{m=0}^{p -1} s_m \mathcal{F}_{m:\infty}(\pv,\pm).
    \end{equation}
\end{lemma}
\noindent The proof of Lemma \ref{lem:periodic_duality} is presented in Appendix \ref{sec:periodic_duality}.
\begin{remark}
    We should note that $\mathcal{P}(\incmat)$ is a partition if $\incmat$ is irreducible. Under the circumstances, if $(\pv,\pm)$ is a $p$-periodic optimizer in Lemma \ref{lem:periodic_duality}, then so is $(\pv,\pm^*)$, where
    \[
    {\pm^*}^{(i)}_{a,b} = {\pm^*}^{(0)}_{a,b} := \pm[j]_{a,b} \quad \text{for all } (a,b) \in \alphabet_{-j-1} \times \alphabet_{-j} \text{ and } i \in [p].
    \]
    Therefore, the value of \eqref{eq:periodic_duality_2} remains unchanged even with every appearance of ``$\max_{\substack{\pm \in \alphsm^p \\ \pv[p] \in \mathcal{C}_0}}$'' replaced by ``$\max_{\substack{(\pv,\pm) \in Z_0: \\ \pm \text{ is } 1\text{-periodic} \\ \pv \text{ is } p\text{-periodic}}}$''.
\end{remark}
For convenience, we introduce the following sequence before we present our quintessential lemma.
\begin{equation} \label{eq:t^*_def}
    t^*_{n,N} = \begin{cases}
        \left(N +\frac{1}{d^p - 1}\right)^{-1} \cdot \frac{d^p}{d^p-1} & \text{if } n = 0; \\
        \left(N +\frac{1}{d^p-1}\right)^{-1} & \text{if } p | n \text{ and } 0 < n < pN ; \\
        0 & \text{otherwise.}
    \end{cases}
\end{equation}
\begin{lemma} \label{lem:periodic_maximality}
    Suppose that $\incmat$ is irreducible and $\rho_{\alphabet_0}(\mathcal{L}_{\incmat,r})$ is as in Proposition \ref{prop:eigenspace}. The following minimax properties hold.
    \begin{align}
        & \min_{r \in \mathcal{R}_{p,d}} \left(\sum_{\ell=0}^{p-1}\prod_{i=0}^{\ell} r_{i}^{-1}\right)^{-1} \cdot \log \rho_{\alphabet_0}(\mathcal{L}_{\incmat,r})\label{eq:spectral_radius_maximality} \\
        = & \lim_{N \to \infty} \min_{s \in \Gamma_{[p]}} \max_{(\pv,\pm) \in Z_0} \sum_{m=0}^{p -1} s_m \sum_{n=0}^{pN-1} t^*_{n,N} \mathcal{F}_{n+m:\infty}(\pv,\pm) \label{eq:covering_maximality} \\
        = & \min_{s \in \Gamma_{[p]}} \max_{\substack{(\pv,\pm) \in Z_0: \\ \pm \text{ is } 1\text{-periodic} \\ \pv \text{ is } p\text{-periodic}}} \sum_{m=0}^{p -1} s_m \mathcal{F}_{m:\infty}(\pv,\pm) \label{eq:irreducible_markov_maximality_dual} \\
        = & \max_{\substack{(\pv,\pm) \in Z_0: \\ \pm \text{ is } 1\text{-periodic} \\ \pv \text{ is } p\text{-periodic}}} \min_{s \in \Gamma_{[p]}} \sum_{m=0}^{p -1} s_m \mathcal{F}_{m:\infty}(\pv,\pm) \label{eq:irreducible_markov_maximality}
    \end{align}
    Furthermore, any maximizer $(\pv^*,\pm^*)$ of \eqref{eq:irreducible_markov_maximality} can be chosen such that $\pm[0]$ is irreducible and ${\pv[0]}^*$ is the unique left eigenvector of $({\pm[0]}^*)^p$ in $\mathcal{C}_0$.
\end{lemma}
\noindent The proof of Lemma \ref{lem:periodic_maximality} is presented in Appendix \ref{sec:periodic_maximality}.
\begin{remark} \label{rem:rho_*_invariant}
    It can be seen, for example, from the cyclic symmetry of $s$ in \eqref{eq:irreducible_markov_maximality_dual} that 
    \[
    \min_{r \in \mathcal{R}_{p,d}} \left(\sum_{\ell=0}^{p-1}\prod_{i=0}^{\ell} r_{i}^{-1}\right)^{-1} \cdot \log \rho_{\alphabet_j}(\mathcal{L}_{\incmat,r})
    \]
    is actually independent of $j$ and therefore further coincides with
    \[
    \min_{r \in \mathcal{R}_{p,d}} \left(\sum_{\ell=0}^{p-1}\prod_{i=0}^{\ell} r_{i}^{-1}\right)^{-1} \cdot \log \rho_*(\mathcal{L}_{\incmat,r}),
    \]
    where $\rho_*$ is defined in Theorem \ref{thm:Hausdorff_dim}.
\end{remark}
As a consequence of Lemma \ref{lem:periodic_maximality}, we have the following corollary.
\begin{corollary} \label{cor:optimal_measure}
    There exists a Markov measure $\mathbb{P}$ associated with an irreducible transition matrix $\transmat$ and an initial distribution $\initvec$ such that it is supported in $\tshift[][\incmat]$ and that, almost surely,
    \begin{equation} \label{eq:optimal_measure}
        \liminf_{n \to \infty} -\frac{\log \mathbb{P}([X|_{\lattice{n}}])}{\norm{\lattice{n}}} = \min_{r \in \mathcal{R}_{p,d}} \left(\sum_{\ell=0}^{p-1}\prod_{i=0}^{\ell} r_{i}^{-1}\right)^{-1} \cdot \log \rho_*(\mathcal{L}_{\incmat,r}),
    \end{equation}
    where $[u]$ denotes the cylinder set associated with the block $u$, namely, given $u \in \block{n}(\tshift[][\incmat])$,
    \begin{equation} \label{eq:cylinder_set}
        [u] := \{t \in \tshift[][\incmat]: t_g = u_g, \forall g \in \lattice{n}\}.
    \end{equation}
\end{corollary}
\begin{proof}
    It straightforwardly follows from Lemma \ref{lem:periodic_maximality} and Theorem \ref{thm:2} by choosing $\obsfunc = \transmat$. More precisely, we may (uniquely) choose the initial distribution as in Lemma \ref{lem:periodic_maximality} so that the initial distribution of a state $a$ is positive if and only if $a \in \alphabet_0$. The proof is then concluded by Theorem \ref{thm:2}.
\end{proof}

\subsection{Lower bound when $\incmat$ is irreducible} \label{sec:lower_bound}
Suppose that $\tshift[][\incmat]$ is a tree-shift with an irreducible incidence matrix $\incmat$ and that $\mathbb{P}$ is the Markov measure in Corollary \ref{cor:optimal_measure}. The lower bound is obtained by applying the mass distribution principle: Due to \eqref{eq:optimal_measure}, for every $\delta < \min_{r \in \mathcal{R}_{p,d}} \left(\sum_{\ell=0}^{p-1}\prod_{i=0}^{\ell} r_{i}^{-1}\right)^{-1} \cdot \log \rho_*(\mathcal{L}_{\incmat,r})$
we have 
\[
\limsup_{n \to \infty} \frac{\mathbb{P}([X|_{\lattice{n}}])}{e^{-\delta\norm{\lattice{n}}}} = 0 \quad \text{almost surely}.
\]
Hence, by Egorov's theorem, there exists a subset $S \subset \tshift[][\incmat]$ with $\mathbb{P}(S) > 0$ and a constant $C$ such that
\[
\mathbb{P}([t|_{\lattice{n}}]) \le C e^{-\delta \norm{\lattice{n}}} \text{ for all } n \in \mathbb{N}, t \in S.
\]
Consequently, supposing that $\mathcal{S}$ is a cover of $S$ consisting of disjoint cylinder sets, we have
\[
\sum_{[u] \in \mathcal{S}} (\mathrm{diam}([u]))^{\delta} \ge C^{-1} \sum_{[u] \in \mathcal{S}} \mathbb{P}([u]) \ge C^{-1} \mathbb{P}(S) > 0.
\]
Therefore, the Hausdorff measure $H^{\delta}(S) \ge C^{-1} \mathbb{P}(S) > 0$ is bounded from below and $\dim_H \tshift[][\incmat] \ge \dim_H S \ge \delta$, implying $\dim_H \tshift[][\incmat] \ge \min_{s \in \mathcal{R}_{p,d}} \left(\sum_{\ell=0}^{p-1}\prod_{i=0}^{\ell} r_{i}^{-1}\right)^{-1} \cdot \log \rho_*(\mathcal{L}_{\incmat,r})$.

\subsection{Upper bound when $\incmat$ is irreducible} \label{sec:upper_bound}
The upper bound of the Hausdorff dimension is proved by constructing suitable covers. As will be seen shortly, the study of the operator $\mathcal{L}_{\incmat,r}$ is motivated by the following na\"ive covering strategy for the space $\tshift[][\incmat]$. For convenience, we denote the collection of all $k$-cylinder sets by $\mathcal{C}_k := \{[t|_{\lattice{n}}]: t \in \tshift[][\incmat]\}$.

Suppose $n + N > n \gg 0$ ($N \in \mathbb{N}$). We define a cover $\mathcal{S}_{n,N}$ agreeing with the following philosophy:
\begin{enumerate}[(a)]
    \item The cover $\mathcal{S}_{n,N}$ consists of cylinder sets in $\cup_{k=n}^{n+N} \mathcal{C}_{k}$.
    \item A cylinder set $[t|_{\lattice{k}}]$ is contained in $\mathcal{S}_{n,N}$ if
    \begin{equation}
        \begin{aligned}
            \frac{\log \norm{\block{k}(\tau(t),\eta(t))}}{\norm{\lattice{k}}} = \min_{n \le m \le n+N} \frac{\log \norm{\block{m}(\tau(t),\eta(t))}}{\norm{\lattice{m}}}.
        \end{aligned}
    \end{equation}
\end{enumerate}
Condition (a) is based on the belief that the Hausdorff dimension is approached by the intermediate dimension (see \cite{banaji2021intermediate}) and condition (b) is imposed in the hope that the upper bound of the Hausdorff dimension is minimized. Precisely, if we denote 
\begin{equation} \label{eq:block_minimax}
    \alpha_{n,N} := \max_{(\tv,\tm) \in \proddom*{n+N}} \min_{n \le m \le n+N} \frac{\log \norm{\block{m}(\tv,\tm)}}{\norm{\lattice{m}}},
\end{equation}
and assume $\alpha > \limsup_{N \to \infty} \limsup_{n \to \infty} \alpha_{n,N}$, then the cover $\mathcal{S}_{n,N}$ gives the following upper bound for the $\alpha$-Hausdorff measure of $\tshift[][\incmat]$:
\begin{align*}
    H^s(\tshift[][\incmat]) & \le \limsup_{N \to \infty} \limsup_{n \to \infty} \sum_{k=n}^{n+N} \sum_{[u] \in \mathcal{C}_k \cap \mathcal{S}_{n,N}} (\mathrm{diam}([u]))^{\alpha} \\
    & = \limsup_{N \to \infty} \limsup_{n \to \infty} \sum_{k=n}^{n+N} \sum_{(\tv,\tm) \in \proddom*{n+N}}\sum_{\substack{[u] \in \mathcal{C}_k \cap \mathcal{S}_{n,N}:\\u \in \block{k}(\tv[0:k],\tm[0:k-1])}} e^{-\alpha \norm{\lattice{k}}} \\
    & \le \limsup_{N \to \infty} \limsup_{n \to \infty} \sum_{k=n}^{n+N} \norm{\proddom*{n+N}} e^{-\left(\alpha-\alpha_{n,N}\right) \norm{\lattice{k}}}.
\end{align*}
Now that Lemma \ref{lem:dist_set_growth_rate} asserts
\[
\norm{\proddom*{n+N}} \le \norm{\dvset{n+N}} \norm{\trmset{n+N}} \le \left(\frac{\norm{\lattice{n+N}}}{n+N+1} + 1\right)^{3 (n+N+1) \cdot \norm{\alphabet}^2},
\]
we deduce that $\tshift[][\incmat]$ is $s$-Hausdorff null:
\begin{align*}
    H^s(\tshift[][\incmat]) &\le \limsup_{N \to \infty} \limsup_{n \to \infty} \sum_{k=n}^{n+N} e^{-\left(\alpha-\alpha_{n,N}-O\left(\frac{\log \norm{\lattice{n+N}}}{\norm{\lattice{n}}}\right)\right) \norm{\lattice{k}}} = 0,
\end{align*}
which implies $\dim_H \tshift[][\incmat] \le \alpha$.

In practice, obtaining an asymptotic estimate for $\alpha_{n,N}$ is facilitated by the following simplification. For each $n,N \in \mathbb{N}$, decompose $\proddom*{n+N}$ into subsets 
\[
\proddom*{n+N}^{(j)}:=\{(\tv,\tm) \in \proddom*{n+N}: \tv[0] \in \mathcal{C}_{j-n-N}\} \text{ for } j \in [p].
\]
Since it is a finite partition, we may assume, by a proper choice of $a_0$ in assumption \eqref{eq:tree_assumption}, that the maximum of $\alpha_{n,N}$ is attained in $\proddom*{n+N}^{(0)}$ for infinitely many $n, N \in \mathbb{N}$, resulting in
\[
    \limsup_{N \to \infty} \limsup_{n \to \infty} \alpha_{n,N} = \limsup_{N \to \infty} \limsup_{n \to \infty} \max_{(\tv,\tm) \in \proddom*{n+N}^{(0)}} \min_{n \le m \le n+N} \frac{\log \norm{\block{m}(\tshift[][\incmat];\tv,\tm)}}{\norm{\lattice{m}}}.
\]
Utilizing Lemma \ref{lem:dense_pattern}, we deduce that for each $n, N \in \mathbb{N}$,
\begin{align*}
    & \max_{(\tv,\tm) \in \proddom*{n+N}^{(0)}} \min_{n \le m \le n+N} \frac{\norm{\block{m}(\tshift[][\incmat];\tv,\tm)}}{\norm{\lattice{m}}} \le \min_{n \le m \le n+N} \max_{(\tv,\tm) \in \proddom*{n+N}^{(0)}} \frac{\norm{\block{m}(\tshift[][\incmat];\tv,\tm)}}{\norm{\lattice{m}}} \\
    \le & \min_{s \in \Gamma_{[N]}} \left[\max_{(\pv,\pm) \in Z_0} \sum_{m=0}^{N} s_m \mathcal{F}_{m:n+N}(\pv,\pm)\right] + o_n(1),
\end{align*}
where $o_n(1)$, according to Lemmas \ref{lem:dist_set_growth_rate} and \ref{lem:uniform_equiconvergence}, is a number independent of $N$ that vanishes as $n \to \infty$. Consequently, by Lemma \ref{lem:periodic_maximality} and Remark \ref{rem:rho_*_invariant}, 
\begin{align*}
    \limsup_{N \to \infty} \limsup_{n \to \infty} \alpha_{n,N} \le & \limsup_{N \to \infty} \min_{s \in \Gamma_{[N]}} \max_{(\pv,\pm) \in Z_0} \sum_{m=0}^{N} s_m \mathcal{F}_{m:\infty}(\pv,\pm) \\
    \le & \limsup_{N \to \infty} \min_{s \in \Gamma_{[p]}} \max_{(\pv,\pm) \in Z_0} \sum_{m=0}^{p -1} s_m \sum_{n=0}^{pN-1} t^*_{n,N} \mathcal{F}_{n+m:\infty}(\pv,\pm) \\
    = & \min_{r \in \mathcal{R}_{p,d}} \left(\sum_{\ell=0}^{p-1}\prod_{i=0}^{\ell} r_{i}^{-1}\right)^{-1} \cdot \log \rho_*(\mathcal{L}_{\incmat,s}),
\end{align*}
implying that the upper bound coincides with the lower bound.
\subsection{Upper bound for general $\tshift[][\incmat]$} \label{sec:general_upper_bound}
In this subsection, we derive an upper bound for $\tshift[][\incmat]$ without assumption \eqref{eq:tree_assumption}. 
\begin{proof}[Proof of Corollary \ref{cor:upper_bound_spectral_radius}]
    Let $t^*_{n,N}$ be defined as in \eqref{eq:t^*_def} with $p = 1$. Let $\alpha_{n,N}$ be as in \eqref{eq:block_minimax} so that, as shown in Section \ref{sec:upper_bound}, $\dim_H \tshift[][\incmat] \le \limsup_{N \to \infty} \limsup_{n \to \infty} \alpha_{n,N}$ and
    \begin{align*}
        & \limsup_{N \to \infty} \limsup_{n \to \infty} \alpha_{n,N} \le \limsup_{N \to \infty} \min_{s \in \Gamma_{[N]}} \left[\max_{(\pv,\pm) \in Z} \sum_{m=0}^{N} s_m \mathcal{F}_{m:\infty}(\pv,\pm)\right]\\
        \le & \limsup_{N \to \infty} \max_{(\pv,\pm) \in Z} \sum_{m=0}^{N} t^*_{m,N} \mathcal{F}_{m:\infty}(\pv,\pm).
    \end{align*}
    Similar to Lemma \ref{lem:asymptotic_property}, the last expression above admits a rearranged form as \eqref{eq:rearranged_form}, which can be solved similarly to \eqref{eq:finite_maximizer} (with $q_{i,N} = (N+\frac{1}{d-1})$ for $0 \le i < N$). Explicitly, this yields
    \begin{align*}
        &\limsup_{N \to \infty} \max_{(\pv,\pm) \in Z} \sum_{m=0}^{N} t^*_{m,N} \mathcal{F}_{m:\infty}(\pv,\pm) \le \limsup_{N \to \infty} \frac{1}{N+\frac{1}{d-1}} \log \left( \max_{a \in \alphabet} (\incmat^N \onevec)_a\right) = \log \rho(\incmat),
    \end{align*}
    proving the proposed inequality. When $\incmat$ is irreducible, it follows from \cite[Theorem 2.1]{Ban2020b} that the entropy satisfies that 
    \begin{equation} \label{eq:dim_comparison}
    \underline{\dim}_{B} \tshift[][\incmat] = h_{top}(\tshift[][\incmat]) \ge \log \rho(\incmat) \ge \dim_{H} \tshift[][\incmat],
    \end{equation}
    in which the first equality holds if and only if $\sum_{b \in \alphabet} \incmat_{a,b}$ is identical for every $a \in \alphabet$. Under the circumstances, one can simply calculate $h_{top}(\tshift[][\incmat]) = \log \rho(\incmat)$ with $\rho(\incmat) = \sum_{b \in \alphabet} \incmat_{a,b}$ for every $a \in \alphabet$ and guess the spectral radius of the transfer operator $\mathcal{L}_{\incmat,r}$ to be $\rho(\incmat)^{\sum_{\ell=0}^{p-1} \prod_{i=0}^{\ell} r_{i}^{-1}}$, the eigenvalue associated with every eigenvectors $\sum_{a \in \alphabet_j} \stdvec{a} \in \mathcal{C}_j$, $j \in [p]$. Plugging it into the proposed formula \eqref{eq:Hausdorff_dimension} yields that the equalities in \eqref{eq:dim_comparison} hold simultaneously.
\end{proof}

Here we provide an example to illustrate our formula.
\begin{example}
    Let $\tshift[][\incmat]$ be a Markov hom tree-shift over the $3$-tree associated with the incidence matrix $\incmat$, where
    \[
    \incmat = \begin{bmatrix}
    \begin{array}{ccc|ccc|ccc}
          &   &   &   &   &   & 0 & 1 & 1 \\
          & 0 &   &   & 0 &   & 1 & 0 & 0 \\
          &   &   &   &   &   & 1 & 1 & 0 \\
        \hline
        0 & 1 & 0 &   &   &   &   &   &   \\
        1 & 0 & 0 &   & 0 &   &   & 0 &   \\
        1 & 0 & 1 &   &   &   &   &   &   \\
        \hline
          &   &   & 0 & 0 & 1 &   &   &   \\
          & 0 &   & 1 & 0 & 0 &   & 0 &   \\
          &   &   & 1 & 1 & 0 &   &   &  
    \end{array}
    \end{bmatrix}.
    \]
    We compute the Hausdorff dimension numerically by determining the eigenvalues and eigenvectors of the corresponding transfer operator, as prescribed by Proposition \ref{prop:eigenspace}. This is achieved through an iterative process, and the logarithm of the eigenvalue is plotted in Figure \ref{fig:example_eigval}. Numerical results suggest that the optimal scaling vector $s$ is approximately $(0.312, 0.588, 0.010)$ with $\log \rho_*(\mathcal{L}_{\incmat,r(s^*)}) \approx 0.3027$, which is strictly less than the logarithm of the spectral radius of $\log \rho(\incmat^T) \approx 0.3208$ and is consistent with Corollary \ref{cor:upper_bound_spectral_radius}.
    \begin{figure}
        \centering
        \includegraphics[width=0.4\textheight]{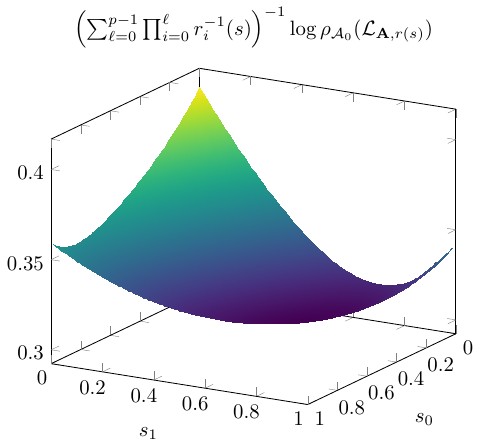}
        \caption{spectral radius $\left(\sum_{\ell=0}^{p-1}\prod_{i=0}^{\ell} r_{i}^{-1}(s)\right)^{-1} \log \rho_{\alphabet_0}(\mathcal{L}_{\incmat,r(s)})$}
        \label{fig:example_eigval}
    \end{figure}
\end{example}

\section{Discussion} \label{sec:discussion}
Inspired by the classical large deviation principle for Markov chains, this article analogously establishes the large deviation principle for irreducible Markov chains indexed by rooted $d$-trees. Building on the belief of its potential for providing a better conception of the tree-shifts as well as further applications, the expositions in this paper showcase the capability of the method of types argument, which not only provides a theorem regarding almost sure convergence but also aids the determination of Hausdorff dimension of the set Markov hom tree-shift. Despite these accomplishments, our work is not intended to be an in-depth investigation but rather a starting point for further exploration. Indeed, several questions still remain unanswered in this work, which we list as follows:
\begin{itemize}
    \item Is the rate function $\Lambda^*_j(\alpha)$ in Theorem \ref{thm:1} continuously differentiable if $\incmat$ is irreducible?
    \item If $\incmat$ is irreducible, can the optimal $\mathcal{L}_{\incmat,r}$ in Theorem \ref{thm:2} be determined? Can the optimal Markov measure in Lemma \ref{lem:periodic_maximality} be determined?
    \item What is the Hausdorff dimension for general $\tshift[][\incmat]$ without the assumption of $\incmat$ being irreducible?
\end{itemize}
We hope this paper will inspire further research in this direction, advancing our understanding of tree-indexed Markov chains and tree-shifts.

\appendix
\section{Proofs of auxiliary lemmas}
\subsection{Proof of Lemma \ref{lem:periodic_duality}} \label{sec:periodic_duality}
    Due to the continuity of the objective function, it suffices to show the equality with $\inf_{s \in \mathring{\Gamma}_{[p]}}$ in place of all appearances of $\min_{s \in \Gamma_{[p]}}$ in \eqref{eq:periodic_duality}. Under the circumstances, we first rephrase, by virtue of periodicity, the left-hand side of \eqref{eq:periodic_duality} as 
    \begin{equation} \label{eq:periodic_duality_2}
        \begin{aligned}
            \max_{\substack{\pm \in \alphsm^p \\ \pv[p] \in \mathcal{C}_0}} \inf_{\substack{s \in \mathring{\Gamma}_{[p]} \\ \mu \in \mathbb{R}^{\alphabet}}} & \left[-\sum_{i=0}^{p-1} \frac{\sum\limits_{m=1}^{p} s_{m+i} d^{m}}{\sum\limits_{m=1}^{p} d^{m}} \pv[p] \left(\prod_{\ell=1}^{p-i-1} \pm[p-\ell]\right) \DKL(\pm[i] || \incmat)\right. \\
            & \quad + \left.\vphantom{\frac{\sum\limits_{m=1}^{p} s_{m+i} d^{m}}{\sum\limits_{m=1}^{p} d^{m}}} \left(\pv[p] \prod_{\ell=1}^{p-1} \pm[p-\ell]  - \pv[p]\right) \mu\right],
        \end{aligned}
    \end{equation} 
    where $s_{p+i}$ are simply aliases of $s_{i}$ for $i=0,\cdots,p-1$, respectively. For brevity, we denote by $F(\pv[p],\pm,s,\mu)$ the objective function above.
    
    As always, the expression on the left of \eqref{eq:periodic_duality} is by definition no larger than the right, so it suffices to show the other inequality. The proof essentially exploits the following sequence of functions $(\lambda^{(i)}(s,\mu))_{i=0}^{p-1}$, whose convexity is guaranteed in Lemma \ref{lem:convexity_backbone}. We let $q_i(s)=(\sum_{m=1}^{p} d^{m})^{-1}(\sum_{m=1}^{p} s_{m+i} d^{m})$, and define
    \[
        \lambda^{(i+1)}(s,\mu) = \begin{cases}
        \mu & \text{if } i=0; \\
            q_i(s) \log (\incmat e^{q_i(s)^{-1} \lambda^{(i)}(s,\mu)}) & \text{if } i=1,2,\cdots,p-1.
        \end{cases}
    \]
    One can now apply the minimax theorem to $\pm[0]$ and $(s,\mu)$, since the objective function is concave in the former variable and convex in the latter. Swapping $\max_{\pm[0]}$ and $\inf_{(s,\mu)}$ turns \eqref{eq:periodic_duality_2} into
    \[
        \begin{aligned}
            \max_{\substack{\pm[1:p-1] \in \alphsm^p \\ \pv[p] \in \mathcal{C}_0}} \inf_{\substack{s \in \mathring{\Gamma}_{[p]} \\ \mu \in \mathbb{R}^{\alphabet}}} & \left[-\sum_{i=1}^{p-1} \frac{\sum\limits_{m=1}^{p} s_{m+i} d^{m}}{\sum\limits_{m=1}^{p} d^{m}} \pv[p] \left(\prod_{\ell=1}^{p-i-1} \pm[p-\ell]\right) \DKL(\pm[i] || \incmat)\right. \\
            & \quad \left.\vphantom{\frac{\sum\limits_{m=1}^{p} s_{m+i} d^{m}}{\sum\limits_{m=1}^{p} d^{m}}}+ \pv[p] \left(\prod_{\ell=1}^{p-1} \pm[p-\ell]\right) \lambda^{(1)}(s,\mu) - \pv[p] \mu \right],
        \end{aligned}
    \]
    Recursively using the convexity of $\lambda^{(i)}$ and applying the minimax theorem, one will move all $\max$ in \eqref{eq:periodic_duality_2} to the right of $\inf$ while retaining equality. In particular, 
    \begin{equation*}
        \begin{aligned}
            & \max_{\substack{\pm \in \alphsm^p \\ \pv[p] \in \mathcal{C}_0}} \inf_{\substack{s \in \mathring{\Gamma}_{[p]} \\ \mu \in \mathbb{R}^{\alphabet}}} F(\pv[p],\pm,s,\mu) = \inf_{\substack{s \in \mathring{\Gamma}_{[p]} \\ \mu \in \mathbb{R}^{\alphabet}}} \max_{\substack{\pm \in \alphsm^p \\ \pv[p] \in \mathcal{C}_0}} F(\pv[p],\pm,s,\mu) \\
            \ge & \inf_{\substack{s \in \mathring{\Gamma}_{[p]} \\ \mu \in \mathbb{R}^{\alphabet}}} \max_{\substack{\pm \in \alphsm^p \\ \pv[p] \in \mathcal{C}_0 \\ \pv[p] = \pv[p] \prod_{\ell=1}^{p-1} \pm[p-\ell]}} F(\pv[p],\pm,s,\mu) \\
            =& \inf_{s \in \mathring{\Gamma}_{[p]}} \max_{\substack{\pm \in \alphsm^p \\ \pv[p] \in \mathcal{C}_0 \\ \pv[p] = \pv[p] \prod_{\ell=1}^{p-1} \pm[p-\ell]}} F(\pv[p],\pm,s,0) = \min_{s \in \Gamma_{[p]}} \max_{\substack{(\pv,\pm) \in Z_0: \\ p\text{-periodic}}} \sum_{m=0}^{p -1} s_m \mathcal{F}_{m}(\pv,\pm).,
        \end{aligned}
    \end{equation*} 
    which agrees with the right-hand side of \eqref{eq:periodic_duality} (with $\inf_{s \in \mathring{\Gamma}_{[p]}}$ in place of $\min_{s \in \Gamma_{[p]}}$.) The lemma is hence proved.
\subsection{Proof of Lemma \ref{lem:periodic_maximality}} \label{sec:periodic_maximality} 
For the sake of simplicity, we denote by $F_N(\pv,\pm,s)$ the objective function of \eqref{eq:covering_maximality} throughout the subsection.
\subsubsection{Auxiliary lemmas}
We first prove the following two auxiliary lemmas. 
\begin{lemma} \label{lem:bijection}
    Suppose $s \in \Gamma_{[p]}$,
    \begin{equation*}
        q_{i}(s) = \sum\limits_{j=0}^{p-1}  \frac{s_{i-j}}{d^{j}} \frac{d^p-d^{p-1}}{d^p-1} \text{ and } r(s)=\left(\frac{q_0(s)}{q_1(s)},\frac{q_1(s)}{q_2(s)},\cdots,\frac{q_{p-1}(s)}{q_0(s)}\right).
    \end{equation*}
    Then, $s \mapsto q(s)$ is a bijection between $\Gamma_{[p]}$ and $\{q \in \Gamma_{[p]}:q_0 \le d q_1 \le \cdots
    \le d^{p-1} q_{p-1} \le d^p q_0\}$, and $q(s) \mapsto r(s)$ is a bijection between $\{q \in \Gamma_{[p]}:q_0 \le d q_1 \le \cdots
    \le d^{p-1} q_{p-1} \le d^p q_0\}$ and $\mathcal{R}_{p,d}$.
\end{lemma}
\begin{proof}
    We note that $q_i$ is a bijection between $\Gamma_{[p]}$ and $\{q \in \Gamma_{[p]}:q_0 \le d q_1 \le \cdots
    \le d^{p-1} q_{p-1} \le d^p q_0\}$. Indeed, $s \mapsto q_i(s)$ is an invertible linear transformation whose inverse can be found by Gaussian elimination, and every element in the latter set admits a preimage in the former. It remains to show that $q(s) \mapsto r(s)$ maps $\{q \in \Gamma_{[p]}:q_0 \le d q_1 \le \cdots
    \le d^{p-1} q_{p-1} \le d^p q_0\}$ bijectively into $\mathcal{R}_{p,d}$. It is not hard to see that such a map is well defined and $\mathcal{R}_{p,d}$ is its image. The injectivity follows from the constraint that $\sum_{i=0}^{p-1} q_{i}(s)=1$, while bijectivity follows by noting that its inverse can be explicitly calculated as $q_0(s) = (\sum_{i=0}^{p-1} \prod_{\ell=0}^{i} r_{\ell}^{-1}(s))^{-1}$ and $q_i(s) = \prod_{\ell=0}^{i-1} r_{\ell}^{-1}(s) q_0(s)$ for $1 \le i < p$.
\end{proof}
\begin{lemma} \label{lem:Lipschitz_error}
    For every $s,s' \in \Gamma_{[p]}$ and every $N \in \mathbb{N}$ 
    \begin{align} \label{eq:independence_identity}
        \left|\max_{(\pv,\pm) \in Z_0} F_{N}(\pv,\pm,s) - \max_{(\pv,\pm) \in Z_0} F_{N}(\pv,\pm,s')\right| \le 2 \cdot \log \norm{\alphabet} \cdot \mathsf{d}_{v}(s,s').
    \end{align}
\end{lemma}
\begin{proof}
It is a consequence of the uniform boundedness $0 \le \mathcal{F}_{n:\infty}(\pv,\pm) \le \log \norm{\alphabet}$.
\end{proof}

\subsubsection{Proof setup} Denote $s_{p+i}=s_{i}$ for $i \in \mathbb{Z}_+$ and set
\[
q_{i,N}(s) := \begin{cases}
    \sum\limits_{j=0}^{i} \frac{s_{i-j}}{d^{j}} \left(N +\frac{1}{d^p-1}\right)^{-1} \frac{d^p-d^{p-1}}{d^p-1} & \text{if } 0 \le i < p; \\
    \sum\limits_{j=0}^{p-1}  \frac{s_{i-j}}{d^{j}} \left(N +\frac{1}{d^p-1}\right)^{-1} \frac{d^p-d^{p-1}}{d^p-1} & \text{if } p \le i < p N; \\
    d^{-(i-pN+1)} q_{i,{pN-1}}(s) & \text{if } p N \le i.
\end{cases}
\]
Multiplying by $(N +\frac{1}{d^p-1})$ and taking the limit $N \to \infty$, we obtain a sequence $(\overline{q}_{i}(s))_{i \in \mathbb{N}}$:
\[
\overline{q}_{i}(s) = \lim_{N \to \infty} (N +\frac{1}{d^p-1}) \cdot q_{i,N}(s),
\]
from which we define a $p$-tuple $r(s)$ that is previously studied in Lemma \ref{lem:bijection}:
\begin{equation} \label{eq:ratio_form}
    r(s)=\left(\frac{\overline{q}_{p}(s)}{\overline{q}_{p+1}(s)},\frac{\overline{q}_{p+1}(s)}{\overline{q}_{p+2}(s)},\cdots,\frac{\overline{q}_{2p-1}(s)}{\overline{q}_{p}(s)}\right).
\end{equation}

Our discussions rely heavily on the following alternative expression of the objective function $F_N$:
\begin{equation} \label{eq:rearranged_form}
    F_N(\pv,\pm,s) = -\sum_{i=0}^{\infty} q_{i,N}(s) \pv[i+1] \DKL(\pm[i] || \incmat).
\end{equation}
Based on this expression, we consider the following finite approximation:
\[
F_{N,m}(\pv,\pm,s) := -\sum_{i=0}^{m-1} q_{i,N}(s) \pv[i+1] \DKL(\pm[i] || \incmat),
\]
which, by Lemma \ref{lem:Lipschitz_error}, admits the following error bound.
\begin{equation} \label{eq:finite_approximation}
    \left|\max_{(\pv,\pm) \in Z_0} F_{N}(\pv,\pm,s) - \max_{(\pv,\pm) \in Z_0} F_{N,pN}(\pv,\pm,s)\right| \le \frac{(d-1) q_{pN-1,N}(s)}{d} \log \norm{\alphabet}.
\end{equation}
Similar to Lemma \ref{lem:asymptotic_property} (or, originally, \cite[Proposition 3.3]{ban2023topological}), the maximum of $F_{N,pN}(\pv,\pm,s)$ can be determined as
\begin{equation} \label{eq:finite_maximizer}
    \max_{(\pv,\pm) \in Z_0} F_{N,pN}(\pv,\pm,s) = \max_{\pv[pN]=\stdvec{a}: a \in \alphabet_0} \pv[pN] \lambda^{(pN),N},
\end{equation}
where
\begin{align*}
    \lambda^{(0),N}(s) = 0 \text{ and } \lambda^{(i+1),N} = \begin{cases}
        q_{i,N}(s) \log \left(\incmat e^{\frac{\lambda^{(i),N}(s)}{q_{i,N}(s)}}\right) & \text{if } q_{i,N}(s) \ne 0, \\
        0 & \text{otherwise}.
    \end{cases}
\end{align*}
Notably, for $0 \le i < p N$, the above can be expressed, by taking $\nu^{(i)}(s) = e^{\frac{\lambda^{(i),N}(s)}{q_{i,N}(s)}}$, as
\begin{equation} \label{eq:iteration}
    \nu^{(i)}(s) = \onevec \text{ and } \nu^{(i+1)}(s) = \begin{cases}
        \Psi_{\incmat,\frac{\overline{q}_{i}(s)}{\overline{q}_{i+1}(s)}}\left(\nu^{(i)}(s)\right) & \text{if } q_{i,N}(0) \ne 0, \\
        \onevec & \text{otherwise},
    \end{cases}  
\end{equation}
which is independent of $N$. In particular, this implies that the iteration satisfies
\[
\nu^{(p+i)}(s) = \Psi_{\incmat,r_i(s)}\left(\nu^{(p+i-1)}(s)\right) \quad \text{for } i=0,\cdots, p(N-1)-1,
\]
which in turn yields
\begin{equation} \label{eq:transfer_operator_iteration}
    \nu^{(pN)}(s) = \mathcal{L}_{\incmat,r^*}^{N-1}(\nu^{(p)}(s)).
\end{equation}

\subsubsection{Proof of the lemma}
\begin{proof}[Proof of Lemma \ref{lem:periodic_maximality}]
    We first prove \eqref{eq:covering_maximality}$\le$\eqref{eq:spectral_radius_maximality}. For every $\varepsilon > 0$, choose $r^* \in \mathcal{R}_{p,d}$ such that
    \[
    \left(\sum_{\ell=0}^{p-1}\prod_{i=0}^{\ell} {r^*_{i}}^{-1}\right) \log \rho_{\alphabet_0}(\mathcal{L}_{\incmat,r^*}) <  \inf_{r \in \mathcal{R}_{p,d}}\left(\sum_{\ell=0}^{p-1}\prod_{i=0}^{\ell} {r_{i}}^{-1}\right) \log \rho_{\alphabet_0}(\mathcal{L}_{\incmat,r}) + \varepsilon,
    \]
    so that, by Lemma \ref{lem:bijection}, we can find some $s^* \in \Gamma_{[p]}$ satisfying $r^* = r(s^*)$. Combining \eqref{eq:finite_approximation}, \eqref{eq:iteration}, \eqref{eq:transfer_operator_iteration}, we have
    \begin{equation} \label{eq:spectral_radius_ge}
        \begin{aligned}
            &~\limsup_{N \to \infty} \max_{(\pv,\pm) \in Z_0} F_{N}(\pv,\pm,s^*) = \limsup_{N \to \infty} 
            \max_{\pv[pN]=\stdvec{a}: a \in \alphabet_0} \pv[pN] \lambda^{(pN),N}(s^*)_a \\
            = &~ \limsup_{N \to \infty} \frac{\overline{q}_{p}(s)}{N+\frac{1}{d^p-1}} \log \left( \max_{a \in \alphabet_0} \mathcal{L}_{\incmat,r^*}^{N-1}(\nu^{(p)}(s))\right),
        \end{aligned}
    \end{equation}
    which, by Proposition \ref{prop:eigenspace} (2), is bounded from above by
    \begin{align*}
        &~ \overline{q}_{p}(s^*) \log \rho_{\alphabet_0}(\mathcal{L}_{\incmat,r^*}) = \left(\sum_{\ell=0}^{p-1}\prod_{i=0}^{\ell} {r^*_{i}}^{-1}\right)^{-1} \log \rho_{\alphabet_0}(\mathcal{L}_{\incmat,r^*}) \\
        \le &~ \inf_{r \in \mathcal{R}_{p,d}}\left(\sum_{\ell=0}^{p-1}\prod_{i=0}^{\ell} {r_{i}}^{-1}\right)^{-1} \log \rho_{\alphabet_0}(\mathcal{L}_{\incmat,r}) + \varepsilon,
    \end{align*}
    where the equality in the first line results from the following:
    \[
    \sum_{i=p}^{2p-1} \overline{q}_i(s^*)=1 \text{ and } \overline{q}_{p+i} = {r^*_{i-1}}^{-1} \cdots {r^*_1}^{-1} {r^*_0}^{-1} \overline{q}_{p}. 
    \]
    
    To prove \eqref{eq:covering_maximality}$\ge$\eqref{eq:spectral_radius_maximality}, we note that one can, by compactness of $Z_0 \times \Gamma_{[p]}$ and continuity of $F_N$, always find a sequence $s(N)$ such that
    \[
    \max_{(\pv,\pm) \in Z_0} F_{N}(\pv,\pm,s(N)) = \min_{s \in \Gamma_{[p]}} \max_{(\pv,\pm) \in Z_0} F_{N}(\pv,\pm,s),
    \]
    which, as a consequence of Lemma \ref{lem:Lipschitz_error}, further admits a convergent subsequence $s(N_i) \in \Gamma_{[p]}$ whose limit $s^*$ satisfies
    \begin{align} \label{eq:independence_identity_limit}
        \liminf_{N \to \infty}\min_{s \in \Gamma_{[p]}} \max_{(\pv,\pm) \in Z_0} F_{N}(\pv,\pm,s) = \lim_{i \to \infty} \max_{(\pv,\pm) \in Z_0} F_{N_i}(\pv,\pm,s^*).
    \end{align}
    We henceforth fix $r^* = r(s^*)$, $\lambda^{(i),j} = \lambda^{(i),j}(s^*)$, and $\nu^{(i)} = \nu^{(i)}(s^*)$ as given in \eqref{eq:ratio_form}, \eqref{eq:finite_maximizer}, and \eqref{eq:iteration}, respectively.
    Recall that, as argued in Proposition \ref{prop:eigenspace} (2), $(\mathcal{L}_{\incmat,r^*}(x)_a)_{a \in \alphabet_0}$ depends only on $(x_a)_{a \in \alphabet_0}$. Therefore, by choosing eigenvector $v^{(0)}$ satisfying $(v^{(0)}_a)_{a \in \alphabet_0} \le (\nu^{(p)}_a)_{a \in \alphabet 0}$, we derive the following by the order-preserving property of $\mathcal{L}_{\incmat,r}$:
    \begin{align*} \label{eq:finite_maximizer_A+}
        &\lim_{i \to \infty} \max_{(\pv,\pm) \in Z_0} F_{N_i,pN_i}(\pv,\pm,s^*) = \lim_{i \to \infty} \frac{\overline{q}_{p}(s^*)}{N+\frac{1}{d^p-1}} \log \left( \max_{a \in \alphabet_0} \mathcal{L}_{\incmat,r^*}^{N-1}(\nu^{(p)})\right) \\
        \ge &~ \lim_{i \to \infty} \frac{\overline{q}_{p}(s)}{N+\frac{1}{d^p-1}} \log \left( \max_{a \in \alphabet_0} \mathcal{L}_{\incmat,r^*}^{N-1}(v^{(0)})\right) = \left(\sum_{\ell=0}^{p-1}\prod_{i=0}^{\ell} {r^*_{i}}^{-1}\right)^{-1} \log \rho_{\alphabet_j}(\mathcal{L}_{\incmat,r^*}). \\
    \end{align*}
    Particularly, the argument here also justifies that ``$\min$'' in \eqref{eq:spectral_radius_maximality} is attained: For the very same $r^*$, we see from the part \eqref{eq:covering_maximality}$\le$\eqref{eq:spectral_radius_maximality} that for any $s \in \Gamma_{[p]}$,
    \begin{align*}
        &\left(\sum_{\ell=0}^{p-1}\prod_{i=0}^{\ell} {r^*_{i}}^{-1}\right)^{-1} \log \rho_{\alphabet_0}(\mathcal{L}_{\incmat,r^*}) \le \lim_{i \to \infty} \min_{s' \in \Gamma_{[p]}} \max_{(\pv,\pm) \in Z_0} F_{N_i}(\pv,\pm,s') \\
        \le& \limsup_{N \to \infty} \max_{(\pv,\pm) \in Z_0} F_{N}(\pv,\pm,s) \le \left(\sum_{\ell=0}^{p-1}\prod_{i=0}^{\ell} r_{i}(s)^{-1}\right)^{-1} \cdot \log \rho_{\alphabet_0}(\mathcal{L}_{\incmat,r})
    \end{align*}
    as desired.
    
    We now prove \eqref{eq:covering_maximality}$\le$\eqref{eq:irreducible_markov_maximality_dual}. To begin with, we fix onward $s^*$ to be a minimizer of \eqref{eq:irreducible_markov_maximality_dual} and hence also $r^* = r(s^*)$, $\lambda^{(i),j} = \lambda^{(i),j}(s^*)$, and $\nu^{(i)} = \nu^{(i)}(s^*)$ as given in \eqref{eq:ratio_form}, \eqref{eq:finite_maximizer}, and \eqref{eq:iteration}, respectively. In view of \eqref{eq:finite_approximation}, it suffices to prove the inequality holds with $F_N(\pv,\pm,s^*)$ in \eqref{eq:covering_maximality} replaced by $F_{N,pN}(\pv,\pm,s^*)$. Under the circumstances, for all $N \in \mathbb{N}$, we can choose, as in \cite[Proposition 3.3]{ban2023topological}, a maximizer $(\pv^*,\pm^*)$ of $F_{N,pN}(\pv,\pm,s^*)$ satisfying the following: For all $0 \le i < pN$,
    \[
    {\pm^*}^{(i)}_{a,b} = \begin{cases}
        \frac{\nu^{(i)}_b}{\sum_{c:\incmat_{a,c}=1} \nu^{(i)}_c} & \text{if } \incmat_{a,b}=1, a \in \alphabet_{-i-1};\\
        0 & \text{otherwise},
    \end{cases}
    \]
    We should note that this expression is independent of $N$. Furthermore, without loss of generality, we may assume $\pv[pN]=\stdvec{a_0}$ for all $N$ by passing to subsequence and changing our choice of $a_0 \in \alphabet_0$ if necessary. Next, we construct a periodic pair $(\tv^*,\tm^*)$ in the feasible domain of \eqref{eq:irreducible_markov_maximality_dual}. Suppose $v^{(j)} \in \mathcal{C}_j^+$, $j \in [p]$, are normalized eigenvectors given in Proposition \ref{prop:eigenspace} (b), which we again extend periodically by setting $v^{(j)}=v^{(j+p)}$ for all $j \in \mathbb{Z}$. Notably, due to the uniqueness of the eigenvectors, the following holds for all $j \in \mathbb{Z}$:
    \[
    v^{(j-1)} = \Psi_{\incmat,r_j}(v^{(j)}) / \Vert \Psi_{\incmat,r_j}(v^{(j)}) \Vert.
    \]
    Our $1$-periodic $\tm^* \in \alphsm^{\mathbb{Z}_+}$ is then defined by
    \begin{equation} \label{eq:optimal_transition}
        {\tm^*}^{(i)}_{a,b} = {\tm^*}^{(0)}_{a,b} := \begin{cases}
            \frac{\incmat_{a,b} v^{(j)}_{b}}{\sum_{c:c \in \alphabet_{j-1}} \incmat_{a,c} v^{(j)}_{c}} & \text{ if } a \in \alphabet_{j-1}; \\
            0 & \text{otherwise.} \\
        \end{cases}
    \end{equation}
    As ${\tm^*}^{(i)} = {\tm^*}^{(0)}$ are irreducible, each of the matrix is, by the Perron-Frobenius theorem, associated with a unique probability vector ${\tv^*}^{(i)} \in \alphpv \cap \mathcal{C}^+_j$ satisfying ${\tv^*}^{(i)} = {\tv^*}^{(i+1)} \tm[0]$ for all $i \in \mathbb{Z}_+$. To establish the aforementioned inequality, we need the following uniform estimate for the distance between $(\pv^*,\pm^*)$ and $(\tv^*,\tm^*)$: There exists $C > 0$ and $0 < \theta < 1$, independent of $N$, such that the following holds for all $n \in \mathbb{Z}_+$:
    \[
    \begin{cases}
        \sup_{a \in \alphabet_{-i-1}} \mathsf{d}_v(({\pm^*}^{(i)}_{a,b})_{b \in \alphabet}, ({\tm^*}^{(i)}_{a,b})_{b \in \alphabet}) \le C \theta^{i}, \\
        \mathsf{d}_v({\pv^*}^{(i)}, {\tv^*}^{(i)}) \le C (\theta^{i} + \theta^{pN - i}), \\
        |{\pv^*}^{(i+1)} \DKL({\pm^*}^{(i)}|| \incmat) - {\tv^*}^{(i+1)} \DKL({\tm^*}^{(i)}|| \incmat)| \le C (\theta^{i} + \theta^{pN - i}).
    \end{cases}
    \]
    The first estimate results from Proposition \ref{prop:eigenspace} (3). For the second, note that
    \begin{align*}
        &~\mathsf{d}_v\left({\pv^*}^{(i)}, {\tv^*}^{(i)} \right) = \mathsf{d}_v\left(\stdvec{a_0} \prod_{\ell=1}^{pN-i} {\pm^*}^{(pN-\ell)} , {\tv^*}^{(i)} \right) \\
        \le &~ \mathsf{d}_v\left(\stdvec{a_0}\prod_{\ell=1}^{pN-i} {\pm^*}^{(pN-\ell)}, \stdvec{a_0} \prod_{\ell=1}^{pN-i} {\tm^*}^{(pN-\ell)}\right) + \mathsf{d}_v\left(\stdvec{a_0} ({\tm^*}^{(0)})^{pN-i}, {\tv^*}^{(i)} \right) \\
        \le &~ C \cdot \frac{{\theta}^{i}}{1-\theta} + C' \cdot {\theta'}^{pN-i},
    \end{align*}
    where the first term in the last line follows from the first estimate and the second from the Perron-Frobenius theorem. Finally, the last estimate is a consequence of the former together with the observation $x \mapsto x \log x$ is Lipschitz on any compact subinterval of $(0,1)$. To complete the proof, note that $q_{i,N}(s^*) = O(1/N)$ and thus
    \begin{align*}
        &~ \norm{F_{N,pN}(\tv^*,\tm^*,s^*) - F_{N,pN}(\pv^*,\pm^*,s^*)} \le \sum_{i=0}^{pN-1} q_{i,N}(s^*) C (\theta^{i} + \theta^{pN - i}) \to 0 \text{ as } N \to \infty.
    \end{align*}

    The proof of the remaining inequalities is rather straightforward: \eqref{eq:covering_maximality}$\ge$\eqref{eq:irreducible_markov_maximality_dual} follows from definition, while \eqref{eq:irreducible_markov_maximality_dual}$\ge$\eqref{eq:irreducible_markov_maximality} is the weak duality. 
    
    To prove the irreducibility of optimal transition matrix, we first note that if $\incmat' \prec \incmat$, then $\mathcal{L}_{\incmat', r}$ still maps $\mathcal{C}_j$ into $\mathcal{C}_j$, and thus $\rho_{\alphabet_j}(\mathcal{L}_{\incmat', r})$ in Proposition \ref{prop:eigenspace} (2) are still well-defined. We then suppose that $(\pv^*,\pm^*,s^*)$ is an optimizer of \eqref{eq:irreducible_markov_maximality}, $r^*$ is a minimizer of \eqref{eq:spectral_radius_maximality} and ${v^{(0)}} \in \mathcal{C}_0$ is an eigenvector of $\mathcal{L}_{\incmat,r^*}$. Now, if $\pm[0]_{a',b'} = 0$ yet $\incmat_{a',b'} = 1$ for some $a',b' \in \alphabet$, we denote the incidence matrix of $\pm[0]_{a',b'}$ as $\incmat'$ and write 
    \[
    Z_0' = \{(\pv,\pm) \in Z_0: \pm[i]_{a',b'} = 0 \text{ for all } i \in \mathbb{Z}_+\}.
    \]
    We can then reproduce \eqref{eq:spectral_radius_maximality}$\ge$\eqref{eq:covering_maximality} as before:
    \begin{align*}
        & \inf_{r \in \mathcal{R}_{p,d}} \left(\sum_{\ell=0}^{p-1}\prod_{i=0}^{\ell} r_{i}^{-1}\right)^{-1} \cdot \log \rho_{\alphabet_0}(\mathcal{L}_{\incmat',r}) \ge \limsup_{N \to \infty} \min_{s \in \Gamma_{[p]}} \max_{(\pv,\pm) \in Z_0'} \sum_{m=0}^{p -1} s_m \sum_{n=0}^{pN-1} t^*_{n,N} f_{n+m}(\pv,\pm).
    \end{align*}
    Additionally, we note that \eqref{eq:covering_maximality}$\ge$\eqref{eq:irreducible_markov_maximality_dual}$\ge$\eqref{eq:irreducible_markov_maximality} follows naturally by definition. These altogether imply that
    \begin{equation} \label{eq:contradiction_inequality}
        \begin{aligned}
            & \inf_{r \in \mathcal{R}_{p,d}} \left(\sum_{\ell=0}^{p-1}\prod_{i=0}^{\ell} r_{i}^{-1}\right)^{-1} \cdot \log \rho_{\alphabet_0}(\mathcal{L}_{\incmat',r}) \\
            \ge & \max_{\substack{(\pv,\pm) \in Z_0: \\ \pm \text{ is } 1\text{-periodic} \\ \pv \text{ is } p\text{-periodic}}} \min_{s \in \Gamma_{[p]}} \sum_{m=0}^{p -1} s_m \mathcal{F}_{m:\infty}(\pv,\pm) = \sum_{m=0}^{p -1} s^*_m \mathcal{F}_{m:\infty}(\pv^*,\pm^*) \\
            = & \min_{r \in \mathcal{R}_{p,d}} \left(\sum_{\ell=0}^{p-1}\prod_{i=0}^{\ell} r_{i}^{-1}\right)^{-1} \cdot \log \rho_{\alphabet_0}(\mathcal{L}_{\incmat,r}).
        \end{aligned}
    \end{equation}
    However, since $\incmat$ is irreducible, there exists $n$ such that $\mathcal{L}^{n}_{\incmat',r^*}({v^{(0)}})_a < \rho_{\alphabet_0}(\mathcal{L}_{\incmat',r})^{n} {v^{(0)}_a}$ for all $a \in \alphabet_0$. As a consequence of the Collatz-Wielandt formula \cite[Theorem 5.6.1]{Lemmens2012} we deduce that
    \[
    \rho_{\alphabet_0}(\mathcal{L}_{\incmat',r^*}) = \rho_{\alphabet_0}(\mathcal{L}^n_{\incmat',r^*})^{1/n} \le \max_{a \in \alphabet_0} \left(\frac{\mathcal{L}^{n}_{\incmat',r^*}(v^{(0)})_a}{v^{(0)}_a}\right)^{1/n} < \rho_{\alphabet_0}(\mathcal{L}_{\incmat,r^*}).
    \]
    This implies
    \begin{align*}
        &\inf_{r \in \mathcal{R}_{p,d}} \left(\sum_{\ell=0}^{p-1}\prod_{i=0}^{\ell} r_{i}^{-1}\right)^{-1} \cdot \log \rho_{\alphabet_0}(\mathcal{L}_{\incmat',r}) \le  \left(\sum_{\ell=0}^{p-1}\prod_{i=0}^{\ell} {r^*_{i}}^{-1}\right)^{-1} \cdot \log \rho_{\alphabet_0}(\mathcal{L}_{\incmat',r^*}) \\
        < & \left(\sum_{\ell=0}^{p-1}\prod_{i=0}^{\ell} {r^*_{i}}^{-1}\right)^{-1} \cdot \log \rho_{\alphabet_0}(\mathcal{L}_{\incmat,r^*}) = \min_{r \in \mathcal{R}_{p,d}} \left(\sum_{\ell=0}^{p-1}\prod_{i=0}^{\ell} r_{i}^{-1}\right)^{-1} \cdot \log \rho_{\alphabet_0}(\mathcal{L}_{\incmat,r}),
    \end{align*}
    contradicting \eqref{eq:contradiction_inequality}. Hence, $\pm^* \sim \incmat$ is irreducible. In this case, the only feasible $\pv[0]$ is the left eigenvector of $(\pm[0])^p$ in $\mathcal{C}_0$.
\end{proof}

\section{Packing dimension of $\tshift[][\incmat]$ when $\incmat$ is irreducible} \label{sec:packing_dim}
When the incidence matrix $\mathbf{M}$ is irreducible, the packing dimension satisfies $\dim_P \tshift[][\incmat] = \overline{\dim}_B \tshift[][\incmat] = d \cdot h_{top}(\tshift[][\incmat])$, a result that, to the best of our knowledge, has not been explicitly addressed in the literature. For completeness, we provide a proof below.
\begin{proposition} \label{prop:packing_dim}
    Suppose that $\incmat$ is an irreducible incidence matrix. Then, $\dim_P \tshift[][\incmat] = \overline{\dim}_B \tshift[][\incmat] = d \cdot h_{top}(\tshift[][\incmat])$.
\end{proposition}
\begin{proof}
    Since $\dim_P \tshift[][\incmat] \le \overline{\dim}_B \tshift[][\incmat]$ by definition and $\overline{\dim}_B \tshift[][\incmat] = d \cdot h_{top}(\tshift[][\incmat])$ as discussed in Section \ref{sec:Hausdorff_dimension_irr}, it remains to show $\dim_P \tshift[][\incmat] \ge d \cdot h_{top}(\tshift[][\incmat])$.
    
    We recall that when $\incmat$ is irreducible, there exists $a \in \alphabet$ such that
    \begin{equation} \label{eq:irr_restriction}
        \lim_{n \to \infty} \frac{\log \left|\left\{u \in \block{pn}(\tshift[][\incmat]): u_\epsilon = a\right\}\right|}{\norm{\lattice{pn}}} = h_{top}(\tshift[][\incmat]),
    \end{equation}
    where $p$ is the common period of the states. This result appears as early as in \cite[Proposition 3.1]{Petersen2018b} and follows also as a consequence of \cite[Theorem 3.1]{ban2023topological}. We henceforth fix $a_0$ in assumption \eqref{eq:tree_assumption} to be one satisfying \eqref{eq:irr_restriction}. The remainder of the argument is based on the mass distribution principle and is somewhat similar to Section \ref{sec:lower_bound}.
    
    Due to irreducibility, there exists $n_0 \in \mathbb{N}$ and a collection $\mathcal{B}=\{(b^{a,i})_{0 \le i \le p n_0}$: $a \in \alphabet\}$ of admissible paths from $a$ to $a_0$ by incidence matrix $\incmat$, namely, $b^{a,0} = a$, $b^{a,p n_0} = a_0$, and $\incmat_{b^{a,i},b^{a,i+1}} = 1$ for all $i$. Choose increasing sequences $(\underline{N}_k)_{k \in \mathbb{N}}, (\overline{N}_k)_{k \in \mathbb{N}} \subset \mathbb{N}$ with each term divisible by $p$ such that $\lim_{k \to \infty} \underline{N}_{k+1}/\overline{N}_k = \infty$ and that
    \[
    0 = \overline{N}_0 < \underline{N}_1 < \overline{N}_1 := \underline{N}_1 + p n_0 < \underline{N}_2 < \overline{N}_2 := \underline{N}_2 + p n_0 < \cdots
    \]
    We iteratively define a measure $\mathbb{P}$ over cylinder sets $\mathcal{C}_{\overline{N}_k}$ ($k \in \mathbb{Z}_+$) that concentrates on configurations with states $a_0$ on the boundary:
    \begin{equation} \label{eq:boundary_condition}
        \mathbb{P}\left(\bigcup \left\{[u] \in \mathcal{C}_{\overline{N}_k}: u_g = a_0, \forall g \in \level{\overline{N}_k}\right\}\right) = 1.
    \end{equation}
    To begin, define $\mathbb{P}$ by $\mathbb{P}([a_0]) = 1$. Suppose that $\mathbb{P}$ is defined for $\mathcal{C}_{\overline{N}_k}$ and satisfies \eqref{eq:boundary_condition}, extend it to a measure over $\mathcal{C}_{\overline{N}_{k+1}}$ by first distributing uniform mass to each admissible configuration in $[u] \in \mathcal{C}_{\underline{N}_{k+1}}$:
    \begin{equation} \label{eq:free_choice_distribution}
        \begin{aligned}
            \mathbb{P}([u]) &= \mathbb{P}([u|_{\lattice{\overline{N}_k}}]) \left|\left\{v \in \block{\overline{N}_k}[\underline{N}_{k+1}]: v_g = a_0, \forall g \in \level{\overline{N}_k} \right\}\right|^{-1} \\
            & = \mathbb{P}([u|_{\lattice{\overline{N}_k}}]) \left(\left|\left\{v \in \block{\underline{N}_{k+1} - \overline{N}_k}: v_\epsilon = a_0\right\}\right|\right)^{-d^{\overline{N}_k}},
        \end{aligned}
    \end{equation}
    and then giving all its mass to the (unique) configuration $w \in \mathcal{C}_{\overline{N}_{k+1}}$ whose paths from level $\underline{N}_{k+1}$ to $\overline{N}_{k+1}$ are of the form of $b^{a,i}$:
    \[
    \mathbb{P}([w]) = \begin{cases}
        \mathbb{P}([w|_{\lattice{\underline{N}_{k+1}}}]) & \text{if } (w_g)_{g' \le g \le g''} \in \mathcal{B}, \forall g' \in \level{\underline{N}_{k+1}}, g'' \in \level{\overline{N}_{k+1}} \\
        0 & \text{otherwise}.
    \end{cases}
    \]
    In this manner, the resultant $\mathbb{P}$ satisfies \eqref{eq:boundary_condition} for configurations in $\mathcal{C}_{\overline{N}_{k+1}}$. This definition uniquely extends $\mathbb{P}$ to a probability measure supported in $\tshift[][\incmat]$. In addition, combining \eqref{eq:free_choice_distribution} with $\lim_{k \to \infty} \underline{N}_{k+1}/\overline{N}_k = \infty$ and \eqref{eq:irr_restriction}, we have almost surely,
    \[
    \limsup_{n \to \infty} \frac{\log \mathbb{P}(B_r(X))}{\log r} = \limsup_{n \to \infty} -\frac{\log \mathbb{P}([X|_{\lattice{n}}])}{\norm{\lattice{n-1}}} = d \cdot h_{top}(\tshift[][\incmat]).
    \]
    
    To conclude the proof, we apply the mass distribution principle to $\mathbb{P}$. Note that $\cup_{i=0}^{\infty} \mathcal{C}_{i}$ is a countable collection of sets and that for two cylinder sets, either they are disjoint or one is contained in the other. Hence, given any $\delta < d \cdot h_{top}(\tshift[][\incmat])$, any countable partition $(S_i)_{i\in\mathbb{N}}$ of $\mathrm{supp}(\mathbb{P})$, and every $\varepsilon > 0$, one can find for each $S_i$ a disjoint countable $\varepsilon$-cover $\mathcal{S}_i$ (and hence an $\varepsilon$-packing) satisfying
    \[
    r^{\delta} \le \mathbb{P}(\overline{B}_r(t)) \quad \text{for all } \overline{B}_r(t) \in \mathcal{S}_i,
    \] 
    where $\overline{B}_r(t)$ denotes the closed $r$-ball centered at $t$. This implies that the packing premeasure $P^{\delta}_0(S_i)$ is bounded from below by $\mathbb{P}(S_i)$:
    \[
    P^{\delta}_0(S_i) \ge \sum_{\overline{B}_r(t) \in \mathcal{S}_i} r^{\delta} \ge \sum_{\overline{B}_r(t) \in \mathcal{S}_i} \mathbb{P}(\overline{B}_r(t)) \ge \mathbb{P}(S_i),
    \]
    and therefore,
    \[
    \sum_{i=1}^{\infty} P^{\delta}_0(S_i) \ge 1.
    \]
    As a result, the packing measure $P^{\delta}(\mathrm{supp}(\mathbb{P}))$ is at least $1$. This proves $\dim_P \tshift[][\incmat] \ge \dim_P \mathrm{supp}(\mathbb{P}) \ge d \cdot h_{top}(\tshift[][\incmat])$.
\end{proof}

\section*{Acknowledgments}
Ban is partially supported by the National Science and Technology Council, R.O.C. (Contract No.~NSTC 111-2115-M-004-005-MY3) and National Center for Theoretical Science. Lai is partially supported by the National Science and Technology Council, R.O.C. (Contract No.~NSTC 111-2811-M-004-002-MY2). The authors extend their gratitude to the anonymous referee for helpful comments that greatly improve the readability of the manuscript.

\bibliographystyle{amsplain_abbrv}
\bibliography{ref}
\end{document}